\documentclass[12pt,reqno]{amsart}
\usepackage{extsizes}
\usepackage{amsmath,tikz}
\usepackage{paralist}
\usetikzlibrary{matrix}
\usepackage{graphics} 
\usepackage{epsfig} 
\usepackage{graphicx}  \usepackage{epstopdf}
\usepackage{amssymb}
\usepackage{amsthm}
\usepackage{epstopdf}
\usepackage{verbatim}
\usepackage{mathrsfs}
\usepackage{mathtools}
\usepackage{pstricks}
\usepackage[colorlinks=true]{hyperref}
\hypersetup{urlcolor=blue, citecolor=red}
\usepackage{subcaption}
\usepackage[percent]{overpic}
\usepackage{longtable}

\newtheorem{theorem}{Theorem}[section]

\newtheorem{lemma}[theorem]{Lemma}
\newtheorem{proposition}{Proposition}

\theoremstyle{definition}

\newtheorem{remark}{Remark}

\title[Rotor imbalance suppression by optimal control] 
{Rotor imbalance suppression by optimal control}

\author[Matteo Gnuffi and Dario Pighin and Noboru Sakamoto]{}


\email{gnuffi@math.unifi.it}
\email{dario.pighin@uam.es}
\email{noboru.sakamoto@nanzan-u.ac.jp}

\thanks{This manuscript has been developed during a secondment in the company ``Marposs S.p.A.'', under the supervision of Matteo Gnuffi and Francesco Ziprani, R\&D Manager. This project has received funding from the European Union’s Horizon 2020 research and innovation programme under the Marie Sklodowska-Curie grant agreement No 777822.  This work has been partially funded by the European Research Council (ERC) under the European Union's Horizon 2020 research and innovation program (grant agreement No. 694126-DyCon), grant MTM2017-82996-C2-1-R of MINECO (Spain), the Marie Curie Training Network ConFlex", the ELKARTEK project KK-2018/00083 ROAD2DC of the Basque Government,  and Nonlocal PDEs: Analysis, Control and Beyond", AFOSR Grant
	FA9550-18-1-0242.\\
	Special thanks go to professor Enrique Zuazua for his suggestions.\\
	The authors acknowledge the engineers Riccardo Cipriani, Stefano Fenara, Andrea Ferrari, Samuele Martelli and Alessandro Ruggeri for their contributions to the manuscript.\\
	We thank professor Davide Barbieri for his support during the secondment.\\
	The authors gratefully acknowledge referees for their interesting remarks.\\
(Noboru Sakamoto) Supported, in part, by JSPS KAKENHI Grant
Number JP19K04446 and by Nanzan University Pache Research Subsidy I-A-2 for 2019 academic year.}

\begin{document}
	\maketitle
	
	\centerline{\scshape Matteo Gnuffi}
	\medskip
	{\footnotesize
		\centerline{Marposs SpA}
		\centerline{40010 Bentivoglio, Italy}
	} 

	\medskip
	
	\centerline{\scshape Dario Pighin$^*$}
	{\footnotesize
		\centerline{Departamento de Matem\'aticas, Universidad Aut\'onoma de Madrid}
		\centerline{28049 Madrid, Spain}
	} 
	\medskip
	{\footnotesize
		\centerline{Chair of Computational Mathematics, Fundaci\'on Deusto}
		\centerline{Avda. Universidades,
			24, 48007, Bilbao, Basque Country, Spain}
	} 

	\medskip

\centerline{\scshape Noboru Sakamoto}
{\footnotesize
	\centerline{Faculty of Science and Engineering, Nanzan University}
	\centerline{Yamazato-cho, Showa-ku, Nagoya, 464-8673, Japan}
} 

	\bigskip
	

	\begin{abstract}
An imbalanced rotor is considered. A system of moving balancing masses is given. We determine the optimal movement of the balancing masses to minimize the imbalance on the rotor. The optimal movement is given by an open-loop control solving an optimal control problem posed in infinite time. By methods of the Calculus of Variations, the existence of the optimum is proved and the corresponding optimality conditions have been derived. Asymptotic behavior of the control system is studied rigorously. By \L ojasiewicz inequality, convergence of the optima as time $t\to +\infty$ towards a steady configuration is ensured. An explicit estimate of the convergence rate is given. This guarantees that the optimal control stabilizes the system. In case the imbalance is below a computed threshold, the convergence occurs exponentially fast. This is proved by the Stable Manifold Theorem applied to the Pontryagin optimality system. Moreover, a closed-loop control strategy based on Reinforcement Learning is proposed. Numerical simulations have been performed, validating the theoretical results.
		
	\end{abstract}
	
	\section{Introduction}
	
    Imbalance vibration affects several rotor dynamic systems. Indeed, often times, rotor's mass distribution is not homogeneous, due to wear, damage and other reasons. The purpose of this paper is to present a control theoretical approach to rotors imbalance suppression. A balancing device, made of moving masses, is given. We look for the optimal movement of a system of balancing masses to minimize the vibrations.

\begin{figure}[htp]
	\begin{center}
		\includegraphics[width=6cm]{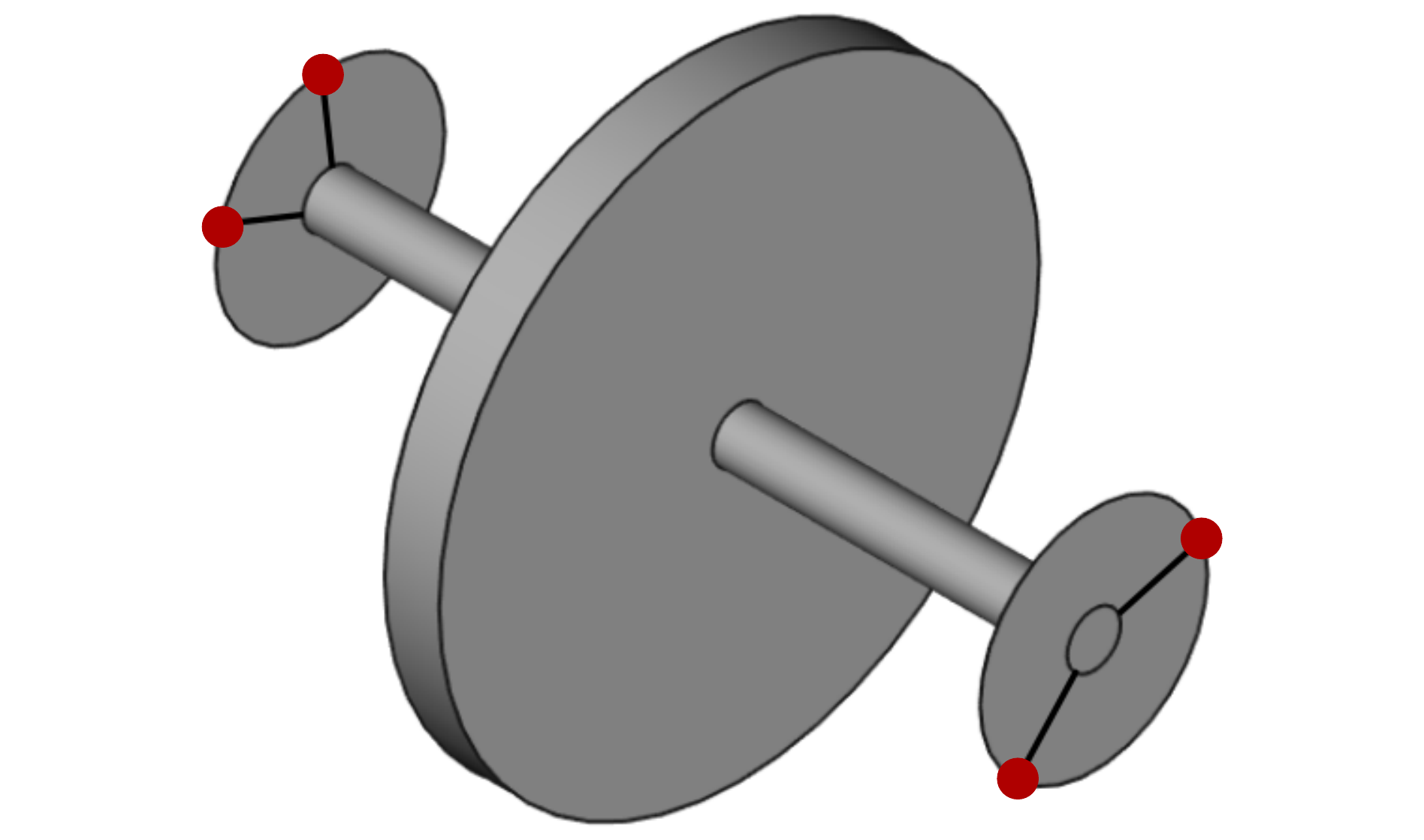}\\
		\caption{The rotor and the balancing device are represented. In the special case represented, the balancing heads are located at the endpoints of the spindle. The four balancing masses (two for each balancing head) are drawn in red.}
		\label{spindlegrinder8mod}
	\end{center}
\end{figure}

The topic is very classical in engineering literature. Indeed, balancing devices are ubiquitous in rotor dynamic systems. 
For instance, grinding machines often get deteriorated during their operational life-cycle. This leads to dangerous imbalance vibrations, which affects their performances while shaping objects (see, for instance,  \cite{HASSUI2003855,HOSHI1986231,ZENG1994495,chen1996grinding}). 
Imbalance is a significant concern for wind turbines as well. In this case, the imbalance may affect the efficiency of power production and the life-cycle of the turbine. If the vibrations become too large, the turbine may collapse. For this reason, vibration detection and correction systems have been developed (see the U.S. patent \cite{jeffrey2012method}). 
Balancing devices have been developed to stabilize CD-ROM drives, washing machines and spacecrafts (see \cite{CHUNG200331,RAJALINGHAM2006169,CHAO20031289,chao2007nonplanar,KIM2005547,xu2015field}). 
Another classical topic in engineering is car's wheels balance. Indeed, easily the wheels can go out of alignment from encountering potholes and/or striking raised objects. Misalignment may cause irregular wear of the tyres. Suspensions components may be damaged as well. For this reason, refined machines have been designed for wheel balancing (see, e.g., \cite[chapter 44]{erjavec2014automotive}).
Vibrations suppression may also involve optimized fluids, like magnetorheological fluids \cite{xu2015field}.
The classical engineering literature on imbalance suppression is concerned with imbalance detection and/or imbalance correction.

In the present work, we address the imbalance correction problem. The imbalance is an input. We consider an imbalanced rotor rotating about a fixed axis at constant angular velocity. We work in the general case of dynamical imbalance, where the imbalanced rotor exert both a force and a torque on the rotation axle. In this context, we suppose that two balancing heads are mounted along two planes orthogonal to the rotation axis. It is assumed that the balancing heads are integral with the rotor, i.e. they rotate together with the rotor. Each balancing head is made of two masses, free to rotate about the rotation axis. Their angular movements are measured with respect to a rotor-fixed reference frame. We can control the balancing masses by actuators transmitting the power with or without contact.\\
An initial configuration of the balancing masses is given. Our goal is to determine four angular trajectories steering the masses from their initial configuration to a steady configuration, where the balancing masses compensate the imbalance. Note that, differently from the classical wheel balancing machines, our balancing device rotates together with the rotor and the rotor is moving while the balancing procedure is accomplished. This motivates us to formulate the problem as a dynamic optimization problem so that transient responses are also taken into account. The method is designed for high speed applications.

A control problem is formulated. We exhibit an open-loop control strategy to move the balancing heads from their initial configuration to a steady configuration, where they compensate the imbalance of the rotor. First of all, viewing the problem in the framework of the Calculus of Variations, the existence of the optimum is proved and the related Euler-Lagrange optimality conditions have been derived. Then, asymptotic behavior of the control system is analyzed rigorously. By \L ojasiewicz inequality, the stabilization of the optimal trajectories towards steady optima is proved. In any condition, explicit bounds of the convergence rate have been obtained. In case the imbalance is below a given threshold, we provide an exponential estimate of the stabilization. Such exponential decay is obtained, by seeing the problem as an optimal control problem, thus writing the Optimality Condition as a first order Pontryagin system. In this context, we prove the hyperbolicity of the Pontryagin system around steady optima, to apply the Stable Manifold Theorem (see \cite[Corollary page 115]{perko2013differential} and \cite{TGS}). Our conclusions fit in the general framework of Control Theory and, in particular, of stabilization, turnpike and controllability (see e.g. \cite{fernandez2003control,sontag1998mathematical,zuazua2007controllability,porretta2013long,trelat2015turnpike,zhu2017geometric,esteve2020turnpike}). In addition, we propose a closed-loop approach using Reinforcement Learning \cite{RDP,sutton2011reinforcement,bokanowski2015value}.

The remainder of the manuscript is organized as follows. In section \ref{sec:1}, we conceive a physical model of the rotor together with the balancing device. In section \ref{sec:2}, we formulate a control problem to determine stabilizing trajectories for the balancing masses. We summarize our achievements in Proposition \ref{prop_group}. The steady problem is analyzed in subsection \ref{subsec:2.2}, where the steady optima are determined. In subsection \ref{subsec:2.3}, we prove some general results. In Proposition \ref{prop3_gen}, the existence of the global minimizer is proved. In Proposition \ref{prop_EL_no_final_condition}, the Optimality Conditions are deduced in the form of Euler-Lagrange equations or equivalently as a state-adjoint state Pontryagin system.
In Proposition \ref{prop4gen} and Proposition \ref{prop_stab_gen_exp} the asymptotic behaviour of the optima is analyzed in the spirit of stabilization and turnpike theory (see \cite{porretta2013long,trelat2015turnpike,TGS,esteve2020turnpike}). The \L ojasiewicz inequality is employed to show that, in any condition, the optima stabilize towards a steady configuration. In case the imbalance does not violate a computed threshold, the stabilization is exponentially fast. This is shown as a consequence of the hyperbolicity of the Pontryagin system around steady optima and the Stable Manifold Theorem. Numerical simulations are performed in subsection \ref{subsec:2.5}. The exponential stabilization of the optima emerges, thus validating the theoretical results. In section \ref{sec:3}, Reinforcement Learning is employed to design a feedback solution. The notation is introduced in table \ref{notation_table}.

\begin{table}
	\centering
	\begin{tabular}{|c|p{380pt}|}
		\multicolumn{2}{|c|}{Notation}\\
		\hline
		\hline
		$\Omega$ & rigid body \\
		\hline
		$\omega$ & angular velocity \\
		\hline
		$(O;(x,y,z))$ & $\Omega$-fixed reference frame \\
		\hline
		$\pi_1$ & first balancing plane \\
		\hline
		$a$ & distance of the first balancing plane from the origin \\
		\hline
		$\pi_2$ & second balancing plane \\
		\hline
		$b$ & distance of the second balancing plane from the origin \\
		\hline
		$m_i$ & mass of balancing masses in $\pi_i$ \\
		\hline
		$P_{i,1}$ & position of the first balancing mass in $\pi_i$ \\
		\hline
		$P_{i,2}$ & position of the second balancing mass in $\pi_i$ \\
		\hline
		$r_i$ & distance from the axle of the balancing masses in $\pi_i$ \\
		\hline
		$b_i$ & the bisector of the angle generated by $\overset{\longrightarrow}{OP_{i,1}}$ and $\overset{\longrightarrow}{OP_{i,2}}$ (see figure \ref{angle_grah_1}) \\
		\hline
		$\alpha_i$ & \textit{intermediate} angle, the angle between the $x$-axis and the bisector $b_i$ \\
		\hline
		$\gamma_i$ & \textit{gap} angle, the angle between $\overset{\longrightarrow}{OP_{i,1}}$ and the bisector $b_i$ \\
		\hline
		$F$ & force exerted by the imbalanced body $\Omega$ on the rotation axis at the origin O \\
		\hline
		$N$ & momentum exerted by the imbalanced body $\Omega$ on the rotation axis, with respect to the pole O \\
		\hline
		$P_1$ & $P_1\coloneqq (0,0,-a)$ intersection of the first balancing plane $\pi_1$ and the $z$ axis \\
		\hline
		$P_2$ & $P_2\coloneqq (0,0,b)$ intersection of the second balancing plane $\pi_2$ and the $z$ axis \\
		\hline
		$F_1$ and $F_2$ & The force $F$ and the momentum $N$ are equivalent to force $F_1$ acting at $P_1$ and force $F_2$ acting at $P_2$ \\
		\hline
		$B_1$ & Balancing force in the first balancing plane $\pi_1$ \\
		\hline
		$B_2$ & Balancing force in the second balancing plane $\pi_2$ \\
		\hline
		$F_{ris,i}$ & $F_{ris,i}=B_i+F_i$ resulting force in $\pi_i$ \\
		\hline
		$G$ & $G\coloneqq\|B_1+F_1\|^2+\|B_2+F_2\|^2$ imbalance indicator, measuring the imbalance on the overall system made of rotor and balancing heads \\
		\hline
		$G_i$ & imbalance indicator for the first balancing plane $\pi_i$ \\
		\hline
		$\Phi_0$ & $\Phi_0=\left(\alpha_{1,0},\gamma_{1,0};\alpha_{2,0},\gamma_{2,0}\right)$ initial configuration for the balancing masses\\
		\hline
		$\Phi(t)$ & $\Phi(t)=\left(\alpha_1(t),\gamma_1(t);\alpha_2(t),\gamma_2(t)\right)$ trajectory for the balancing masses, \textit{state} of the control problem \\
		\hline
		$\psi(t)$ & $\psi(t)\coloneqq (\psi_1(t),\psi_2(t);\psi_3(t),\psi_4(t))$ variable for the time derivative of $\Phi$ \\
		\hline
		$\beta$ & weighting parameter in the cost functional \\
		\hline
		$L$ & Lagrangian \\
		\hline
		$\hat{G}$ & $\hat{G}\coloneqq G-\inf G$ \\
		\hline
		$\mathscr{S}$ & $\mathscr{S}\coloneqq \mbox{argmin}(G)$ set of minimizers of the imbalance indicator $G$ \\
		\hline
		$\mathscr{A}$ & set of admissible trajectories \\
		\hline
		$J$ & cost functional in the control problem \\
		\hline
		$\overline{\Phi}$ & optimal steady state \\
		\hline
	\end{tabular}
	\caption{Notation table. Subsection \ref{subsec:2.3} generalizes the notation to general control problems. Index $i=1,2$.}
	\label{notation_table}
\end{table}

\newpage
	
\section{The model}
\label{sec:1}
Assume the rotor is a rigid body $\Omega\subset\mathbb{R}^3$ rotating about an axis at a constant angular velocity $\omega$. Often times the rotor mass distribution is not homogeneous, producing imbalance in the rotation. This leads to dangerous vibrations. Our goal is to find the optimal movement of a system of balancing masses in order to minimize the imbalance.

Consider $(O;(x,y,z))$ $\Omega$-fixed reference frame.
By definition, the axes $(x,y)$ rotate about axis $z$ at a constant angular velocity $\omega$.

\begin{figure}
	\begin{overpic}[scale=0.35]{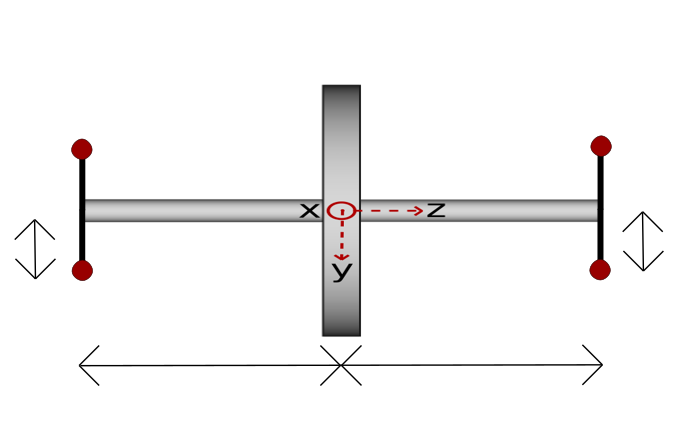}
		\put (-3.85, 24.0) {$r_1$}
		\put (98.85, 24.0) {$r_2$}
		\put (30, 3.0) {$a$}
		\put (68, 3.0) {$b$}
	\end{overpic}
	\caption{Front view of the system made of rotor and balancing device.}
	\label{spindlegrinder7_frontmodmeasnew}
\end{figure}

\begin{figure}
	\centering
	\begin{subfigure}{.5\textwidth}
		\centering
		\includegraphics[width=\textwidth]{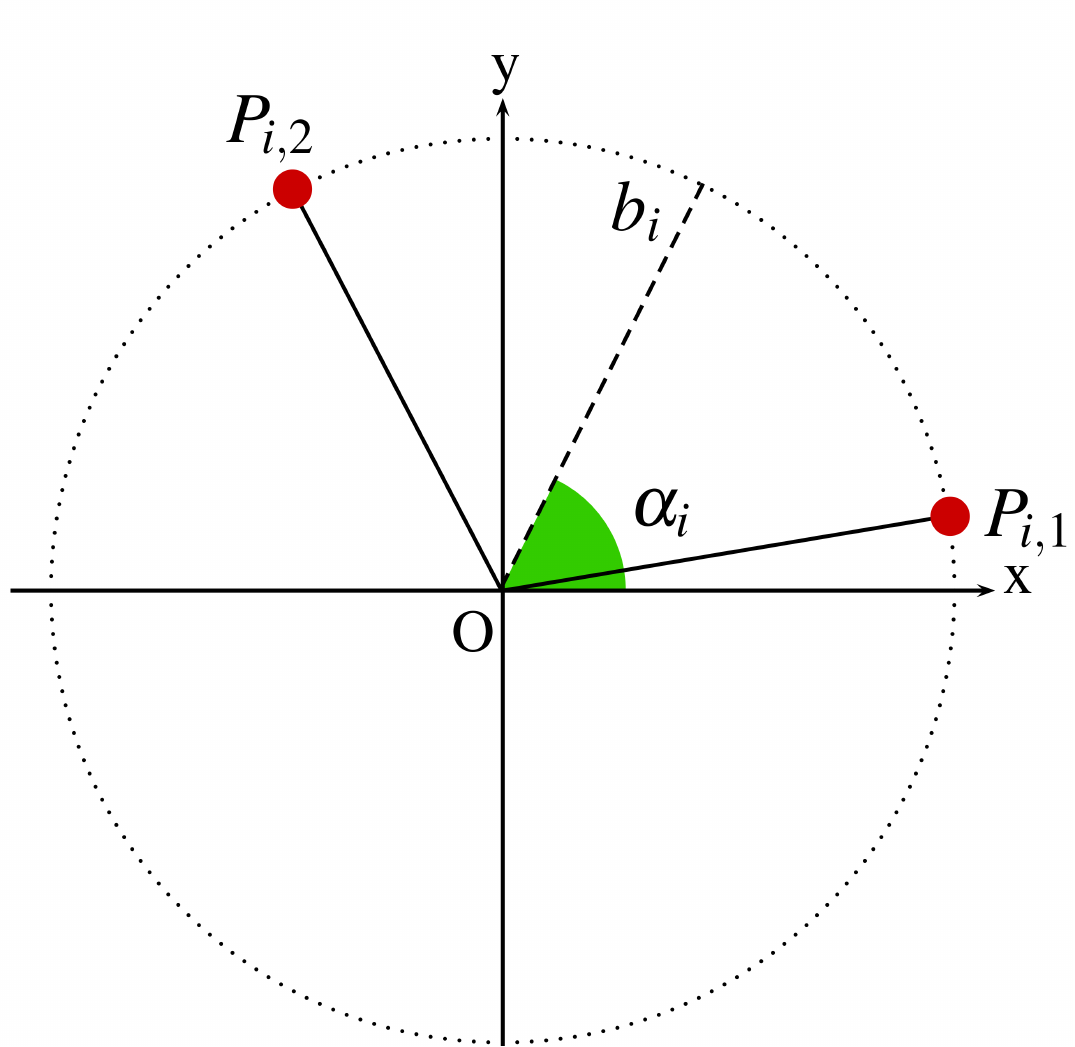}
		\caption{\textit{intermediate} angle}
		\label{fig:sub1}
	\end{subfigure}%
	\begin{subfigure}{.5\textwidth}
		\centering
		\includegraphics[width=\textwidth]{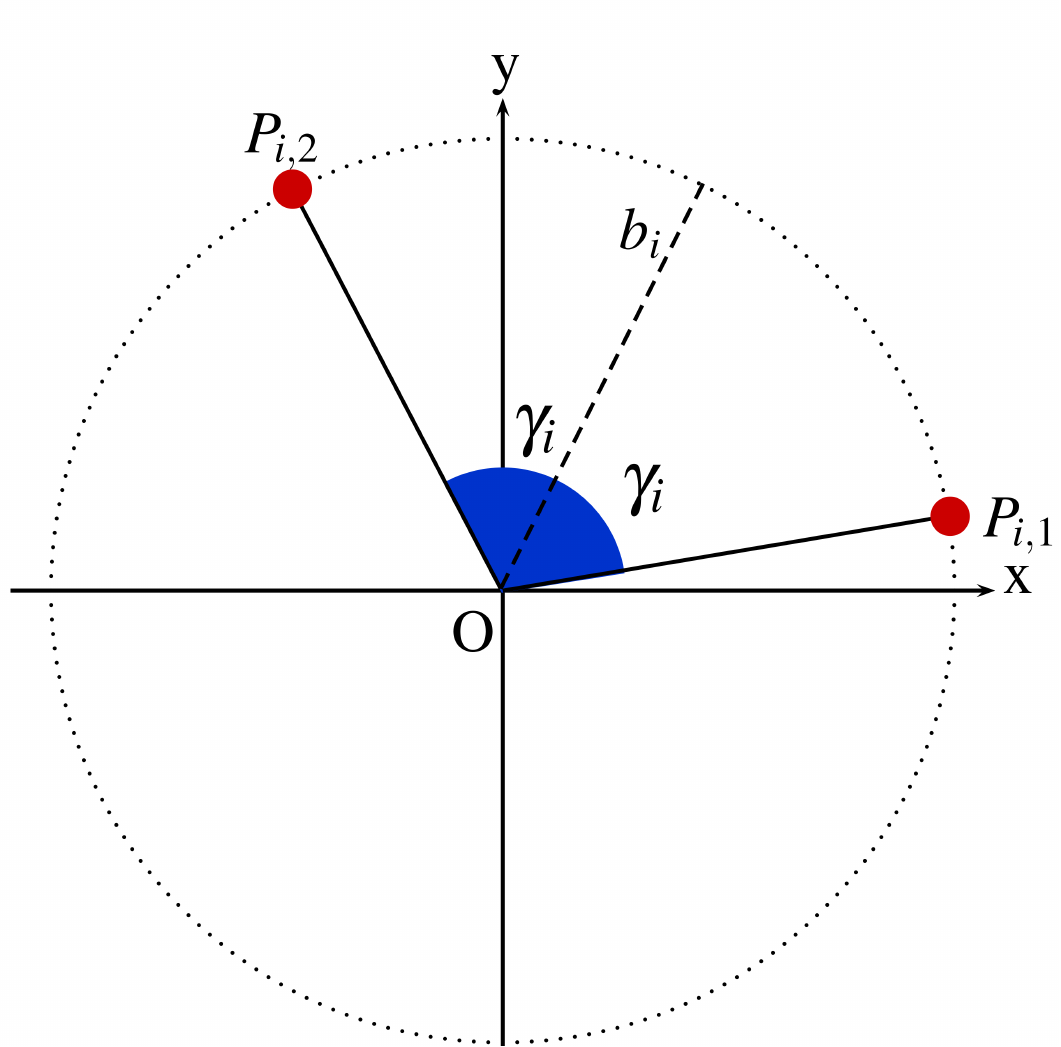}
		\caption{\textit{gap} angle}
		\label{fig:sub2}
	\end{subfigure}
	\caption{One balancing head is considered. The balancing masses $(m_i,P_{i,1})$ and $(m_i,P_{i,2})$ are drawn in red. The bisector of the angle generated by $\overset{\longrightarrow}{OP_{i,1}}$ and $\overset{\longrightarrow}{OP_{i,2}}$ is the dashed line. The \textit{intermediate} angle $\alpha_i$ and the \textit{gap} angle $\gamma_i$ give the position of the balancing masses in each balancing head. For general mathematical notation, we refer to the section at the end of the manuscript.}\label{angle_grah_1}
\end{figure}

The balancing device (see figures \ref{spindlegrinder8mod} and \ref{spindlegrinder7_frontmodmeasnew}) is made up two heads lying in two planes orthogonal to the rotation axis $z$. Each head is made of a pair of balancing masses, which are free to rotate on a plane orthogonal to the rotation axis $z$.  Namely, we have
\begin{itemize}
	\item two planes $\pi_1\coloneqq \left\{z=-a\right\}$ and $\pi_2\coloneqq \left\{z=b\right\}$, with $a$, $b\geq 0$;
	\item two mass-points $(m_1,P_{1,1})$ and $(m_1,P_{1,2})$ lying on $\pi_1$ at distance $r_1$ from the axis $z$, i.e.,\\
	in the reference frame $(O;(x,y,z))$
	\begin{eqnarray}\label{P1}
	\begin{dcases}
	P_{1,1;x}=&r_1\cos(\alpha_1-\gamma_1)\\
	P_{1,1;y}=&r_1\sin(\alpha_1-\gamma_1)\\
	P_{1,1;z}=&-a,\\
	\end{dcases}\nonumber\\
	\mbox{and\hspace{3.5 cm}}\\
	\begin{dcases}
	P_{1,2;x}=&r_1\cos(\alpha_1+\gamma_1)\\
	P_{1,2;y}=&r_1\sin(\alpha_1+\gamma_1)\\
	P_{1,2;z}=&-a;\nonumber\\
	\end{dcases}
	\end{eqnarray}
	\item two mass-points $(m_2,P_{2,1})$ and $(m_2,P_{2,2})$ lying on $\pi_2$ at distance $r_2$ from the axis $z$, namely, in the reference frame $(O;(x,y,z))$
	\begin{eqnarray}\label{P2}
	\begin{dcases}
	P_{2,1;x}=&r_2\cos(\alpha_2-\gamma_2)\\
	P_{2,1;y}=&r_2\sin(\alpha_2-\gamma_2)\\
	P_{2,1;z}=&b,\\
	\end{dcases}\nonumber\\
	\mbox{and\hspace{3.5 cm}}\\
	\begin{dcases}
	P_{2,2;x}=&r_2\cos(\alpha_2+\gamma_2)\\
	P_{2,2;y}=&r_2\sin(\alpha_2+\gamma_2)\\
	P_{2,2;z}=&b.\nonumber\\
	\end{dcases}
	\end{eqnarray}
\end{itemize}
For any $i=1,2$, let $b_i$ be the bisector of the angle generated by $\overset{\longrightarrow}{OP_{i,1}}$ and $\overset{\longrightarrow}{OP_{i,2}}$ (see figure \ref{angle_grah_1}). For any $i=1,2$, the \textit{intermediate} angle $\alpha_i$ is the angle between the $x$-axis and the bisector $b_i$, while the \textit{gap} angle $\gamma_i$ is the angle between $\overset{\longrightarrow}{OP_{i,1}}$ and the bisector $b_i$.
Note that the angles $\alpha_i$ and $\gamma_i$ are defined with respect to the $\Omega$-fixed reference frame $(O;(x,y,z))$. Indeed, the balancing device described above is integral with the body $\Omega$. Furthermore, we observe that on the one hand, in view of avoiding the generation of torque in each single head, the two balancing masses composing a single head are placed on a single plane. On the other hand, the available balancing heads are placed on two separate planes and torque may be generated by the composed action of the heads.

Following a classical approach, the imbalance may be described as the force $F$ and the momentum $N$ exerted by the imbalanced body $\Omega$ on the rotation axis. The force is applied at the origin O. The momentum is computed with respect to the pole O. Both the force and the momentum are supposed to be orthogonal to the rotation axis $z$. As we mentioned, $F$ and $N$ are given data.

In $(O;(x,y,z))$, set $P_1\coloneqq (0,0,-a)$, $P_2\coloneqq (0,0,b)$, $F\coloneqq (F_x,F_y,0)$ and $N\coloneqq (N_x,N_y,0)$. By imposing the equilibrium condition on forces and momenta, the force $F$ and the momentum $N$ can be decomposed into a force $F_1$ exerted at $P_1$ contained in plane $\pi_1$ and a force $F_2$ exerted at $P_2$ contained in $\pi_2$
\begin{equation}\label{F_1}
F_1=
\frac{1}{a+b}\begin{tikzpicture}[baseline=(current bounding box.center)]
\matrix (b) [matrix of math nodes,nodes in empty cells,right delimiter={]},left delimiter={[} ]{
	bF_x-N_y  \\
	bF_y+N_x  \\
	0  \\
} ;
\end{tikzpicture}
\hspace{0.3 cm}\mbox{and}\hspace{0.3 cm}
F_2=
\frac{1}{a+b}\begin{tikzpicture}[baseline=(current bounding box.center)]
\matrix (b) [matrix of math nodes,nodes in empty cells,right delimiter={]},left delimiter={[} ]{
	aF_x+N_y  \\
	aF_y-N_x  \\
	0  \\
} ;
\end{tikzpicture}.
\end{equation}

In each plane, we are able to generate a force to balance the system, by moving the balancing masses described in \eqref{P1} and \eqref{P2}.

In particular, by trigonometric formulas
\begin{itemize}
	\item in plane $\pi_1$, we compensate force $F_1$ by the centrifugal force:
	\begin{equation}\label{B_1}
	B_1=2m_1r_1\omega^2\cos(\gamma_1)\left(\cos(\alpha_1),\sin(\alpha_1)\right);
	\end{equation}
	\item in plane $\pi_2$, we compensate force $F_2$ by the centrifugal force:
	\begin{equation}\label{B_2}
	B_2=2m_2r_2\omega^2\cos(\gamma_2)\left(\cos(\alpha_2),\sin(\alpha_2)\right).	
	\end{equation}
\end{itemize}

The overall imbalance of the system is then given by the resulting force in $\pi_1$
\begin{equation}
F_{ris,1}=B_1+F_1
\end{equation}
and the resulting force in $\pi_2$
\begin{equation}
F_{ris,2}=B_2+F_2.
\end{equation}
Note that, if the balancing masses are moved incorrectly, we may increase the imbalance on the system.

We introduce the imbalance indicator
\begin{equation}\label{G}
G\coloneqq\|B_1+F_1\|^2+\|B_2+F_2\|^2.
\end{equation}
The above quantity measures the imbalance on the overall system made of rotor and balancing heads.


By \eqref{B_1} and \eqref{B_2}, we observe that
\begin{equation}\label{split}
G(\alpha_1,\gamma_1,\alpha_2,\gamma_2)=G_1(\alpha_1,\gamma_1)+G_2(\alpha_2,\gamma_2),
\end{equation}
where
\begin{eqnarray*}
	G_1(\alpha_1,\gamma_1)&\coloneqq&\left[\left|2m_1r_1\omega^2\cos(\gamma_1)\cos(\alpha_1)+F_{1,x}\right|^2\right.\nonumber\\
	&\;&\left.+\left|2m_1r_1\omega^2\cos(\gamma_1)\sin(\alpha_1)+F_{1,y}\right|^2\right]
\end{eqnarray*}
and
\begin{eqnarray*}
	G_2(\alpha_2,\gamma_2)&\coloneqq&\left[\left|2m_2r_2\omega^2\cos(\gamma_2)\cos(\alpha_2)+F_{2,x}\right|^2\right.\nonumber\\
	&\;&\left.+\left|2m_2r_2\omega^2\cos(\gamma_2)\sin(\alpha_2)+F_{2,y}\right|^2\right].
\end{eqnarray*}

\section{The control problem}
\label{sec:2}

An initial configuration $\Phi_0=\left(\alpha_{1,0},\gamma_{1,0};\alpha_{2,0},\gamma_{2,0}\right)$ for the balancing masses is given.\\
Our goal is to find a control strategy such that:
\begin{itemize}
	\item the balancing masses move from $\Phi_0$ to a final configuration $\overline{\Phi}=\left(\overline{\alpha}_1,\overline{\gamma}_1;\overline{\alpha}_2,\overline{\gamma}_2\right)$, where they compensate the imbalance;
	\item the imbalance should not increase and velocities of the masses are kept small during the correction process. 
\end{itemize}
In this first part, we suppose that we do not have a real-time feedback concerning the imbalance of the system. For this reason, we design an open-loop control. A closed-loop strategy is designed in section \ref{sec:3}.

Accordingly, we introduce a control problem to steer our system to a stable configuration, which minimizes the imbalance. In the context of the model described in section \ref{sec:1}, we choose as \textit{state} $\Phi(t)\coloneqq (\alpha_1(t),\gamma_1(t);\alpha_2(t),\gamma_2(t))$, where $\alpha_i(t)$ and $\gamma_i(t)$ are the angles regulating the position of the four balancing masses, as illustrated in \eqref{P1} and \eqref{P2}.\\
The \textit{control} $\psi(t)\coloneqq (\psi_1(t),\psi_2(t);\psi_3(t),\psi_4(t))$ is the time derivative of the state, i.e. its components are the time derivatives of the angles $\Phi_i(t)$. Namely, the state equation is
\begin{equation}\label{}
\begin{dcases}
\frac{d}{dt}\Phi=\psi\hspace{0.6 cm} &t\in (0,+\infty)\\
\Phi(0)=\Phi_0.\\
\end{dcases}
\end{equation}
Note that we are in the particular case of the Calculus of Variations. The time interval is infinite and special attention has to be paid for the limiting behavior of the solution. 

The Lagrangian $L:\mathbb{T}^4\times \mathbb{R}^4\longrightarrow \mathbb{R}$ reads as
\begin{equation}
L\left(\Phi,\psi\right) \coloneqq \frac{1}{2}\left[\|\psi\|^2+{\beta}\hat{G}(\Phi)\right],
\end{equation}
where $\beta>0$ is a parameter to be fixed and $\hat{G}\coloneqq G-\inf G$, $G$ being the imbalance indicator introduced in \eqref{G}. Note that for any $\overline{\Phi}\in\mathscr{S}\coloneqq \mbox{argmin}(G)$, $\hat{G}(\overline{\Phi})=G(\overline{\Phi})-\inf G=\inf G-\inf G=0$, namely $\mathscr{S}$ coincides with the zero set of $\hat{G}$. We have introduced $\hat{G}$ to guarantee the integrability of the Lagrangian along admissible trajectories over the half-line $(0,+\infty)$.\\
In the above Lagrangian, there is a trade-off between the
cost of controlling the system to a stable regime and the velocity of the balancing masses, with respect to the rotor. If $\beta$ is large, the primary concern for the optimal strategy is to minimize the cost of controlling, while if $\beta$ is small our priority is to minimize the velocities.

Let $\Phi_0\in \mathbb{T}^4$ be an initial configuration. We introduce the space of admissible trajectories
\begin{equation}\label{admissible_trajectories}
\mathscr{A}\coloneqq \left\{\Phi\in H^1_{loc}([0,+\infty);\mathbb{T}^4) \hspace{0.3 cm} \big| \hspace{0.3 cm} \Phi(0)=\Phi_0,\hspace{0.3 cm}\mbox{and}\hspace{0.3 cm}L(\Phi,\dot{\Phi})\in L^1(0,+\infty)\right\},
\end{equation}
where the Sobolev space $H^1_{loc}((0,+\infty);\mathbb{T}^4)$ is defined in \eqref{H1loc} (section \ref{sec:Notation}). Note that the requirement $L(\Phi,\dot{\Phi})\in L^1(0,+\infty)$ is equivalent to
\begin{equation}
\dot{\Phi}\in L^2(0,+\infty)\hspace{0.3 cm}\mbox{and}\hspace{0.3 cm}G(\Phi)-\inf G\in L^1(0,+\infty).
\end{equation}

Our goal is to minimize the functional $
J:\mathscr{A} \longrightarrow \mathbb{R}$
\begin{equation}\label{functional}
J\left(\Phi\right) \coloneqq\frac{1}{2}\int_0^{\infty} \left[\|\dot{\Phi}\|^2+{\beta}\hat{G}(\Phi)\right] dt.
\end{equation}

\subsection{Statement of the main result}
\label{subsec:2.1}

We state now our main result.

\begin{proposition}\label{prop_group}
	Consider the functional \eqref{functional}. For $i=1,2$, set
	\begin{equation}\label{c_J_i}
	c^i\coloneqq \frac{1}{{2m_ir_i\omega^2}}\left(F_{i,x},F_{i,y}\right),
	\end{equation}
	where the above notation has been introduced in table \ref{notation_table}. Then,
	\begin{enumerate}
		\item there exists $\Phi\in \mathscr{A}$ minimizer of $J$;
		\item $\Phi=\left(\alpha_1,\gamma_1;\alpha_2,\gamma_2\right)$
		is $C^{\infty}$ smooth and, for $i=1,2$, the following Euler-Lagrange equations are satisfied, for $t>0$
	\end{enumerate}
	\begin{equation}\label{EL_new}
	\begin{cases}
	-\ddot{\alpha}_i=\beta\cos\left(\gamma_i\right)\left[-c^i_1\sin\left(\alpha_i\right)+c^i_2\cos\left(\alpha_i\right)\right]\\
	-\ddot{\gamma}_i=-\beta\sin\left(\gamma_i\right)\left[c^i_1\cos(\alpha_i)+c^i_2\sin(\alpha_i)-\cos(\gamma_i)\right]\\
	\alpha_i(0)=\alpha_{0,i},\hspace{0.16 cm}\gamma_i(0)=\gamma_{0,i}, \hspace{0.16 cm} \dot{\Phi}(T)\underset{T\to +\infty}{\longrightarrow}0.\\
	\end{cases}
	\end{equation}
	\begin{enumerate}
		\item[(3)] for any optimal trajectory $\Phi$ for \eqref{functional}, there exists $\overline{\Phi}\in\mathscr{S}$ such that
		\begin{equation}\label{prop_group_eq1}
		\Phi(t)\underset{t\to +\infty}{\longrightarrow}\overline{\Phi},
		\end{equation}
		\begin{equation}\label{prop_group_eq2}
		\dot{\Phi}(t)\underset{t\to +\infty}{\longrightarrow}0.
		\end{equation}
		and
		\begin{equation}\label{prop_group_eq3}
		\left|\hat{G}\left(\Phi(t)\right)\right|\underset{t\to +\infty}{\longrightarrow}0.
		\end{equation}
		If, in addition
	\end{enumerate}
	\begin{equation}\label{prop_group_eq4}
	m_1r_1 > \frac{\sqrt{F_{1,x}^2+F_{1,y}^2}}{2\omega^2}\hspace{0.3 cm}\mbox{and}\hspace{0.3 cm}m_2r_2> \frac{\sqrt{F_{2,x}^2+F_{2,y}^2}}{2\omega^2},
	\end{equation}
	\begin{enumerate}
		\item[] we have the exponential estimate for any $t\geq 0$
	\end{enumerate}
	\begin{equation}\label{prop_group_eq6}
	\|\Phi(t)-\overline{\Phi}\|+\|\dot{\Phi}(t)\|+\left|G\left(\Phi(t)\right)\right|\leq C\exp\left(-\mu t\right),
	\end{equation}
	\begin{enumerate}
		\item[]
		with $C, \hspace{0.07 cm}\mu >0$ independent of $t$.
	\end{enumerate}
\end{proposition}

In the following subsection, we analyze the corresponding steady problem. In subsection \ref{subsec:2.3}, we develop general tools to prove the above result. In subsection \ref{subsec:2.4}, we prove Proposition \ref{prop_group}. In subsection \ref{subsec:2.5}, we perform some numerical simulations validating the theory. In section \ref{sec:3} we present the feedback strategy.

\subsection{The steady problem}
\label{subsec:2.2}

First of all, we address the \textit{steady} problem:
\begin{itemize}
	\item[]\textit{Find a 4-tuple of angles $(\overline{\alpha}_1,\overline{\gamma}_1;\overline{\alpha}_2,\overline{\gamma}_2)$ such that the imbalance indicator $G$ is minimized}.
\end{itemize}

A solution to the above steady problem is called \textit{steady optimum}. We recall that the set of steady optima is denoted by $\mathscr{S}= \mbox{argmin}\left(G\right)$.

\begin{remark}\label{remark_steadyoptima}
	We observe that by using \eqref{split},
	\begin{equation}
	\mathscr{S}=\mbox{argmin}(G_1)\times \mbox{argmin}(G_2),
	\end{equation}
	namely we can reduce our 4-dimensional problem to a 2-dimensional problem.
\end{remark}

Therefore, we have reduced to find minimizers of a function of the form:
\begin{equation}
g(\alpha,\gamma)\coloneqq \left|\cos(\gamma)\cos(\alpha)-c_1\right|^2+\left|\cos(\gamma)\sin(\alpha)-c_2\right|^2.
\end{equation}
This task is accomplished in Lemma below.

\begin{lemma}\label{lemma_argming}
	Let $c=(c_1,c_2)\in\mathbb{R}^2$. Set
	\begin{equation}\label{g_lemma_argming}
	g(\alpha,\gamma)\coloneqq \left|\cos(\gamma)\cos(\alpha)-c_1\right|^2+\left|\cos(\gamma)\sin(\alpha)-c_2\right|^2.
	\end{equation}
	Let $\mbox{argmin}\left(g\right)$ be the set of minimizers of $g$. Then,
	\begin{enumerate}
		\item if $c=0$, then
		\begin{equation}
		\mbox{argmin}(g)=\left\{\left(\theta,\frac{\pi}{2}\right) \ \bigg| \ \theta\in\mathbb{T}\right\};
		\end{equation}
		\item if $c\neq 0$, set $d\coloneqq \min\left\{1,\|c\|\right\}$. Then,
		\begin{equation}\label{argming}
		\mbox{argmin}(g)=\left(\arg(c_1+ic_2),\arccos(d)\right)\cup
		\end{equation}
		\begin{equation}
		\cup \left(\arg(c_1+ic_2)+\pi,\arccos(-d)\right),
		\end{equation}
		where $\arg(c_1+ic_2)$ denotes the argument of the complex number $c_1+ic_2$.\\
		Moreover, if $c\neq 0$, there exists a unique $(\alpha,\gamma)$ minimizer of $g$, with $ 0\leq \alpha<2\pi$ and $0\leq \gamma\leq \frac{\pi}{2}$;
		\item $\inf g=0$ if and only if $\|c\|\leq 1$;
		\item $\inf g = 
		\begin{dcases}
		0,\hspace{0.6 cm} \hspace{0.3 cm}& \mbox{if} \hspace{0.10 cm}\|c\|\leq 1\\
		\left|\|c\|-1\right|^2,  \hspace{0.3 cm}& \mbox{if} \hspace{0.10 cm}\|c\|> 1.
		\end{dcases}$
	\end{enumerate}
\end{lemma}
This Lemma can be proved by trigonometric calculus.

Now, let $\overline{\Phi}\in\mathscr{S}$ be a minimizer of the imbalance indicator $G$. We highlight that two circumstances may occur:
\begin{itemize}
	\item $\inf G=0$, namely, the overall system made of rotor and balancing masses can be fully balanced, by placing the four balancing masses as
	\begin{eqnarray}\label{steady_config_head1}
	P_{1,1}&=&r_1\left(\cos\left(\overline{\alpha}_1-\overline{\gamma}_1\right),\sin\left(\overline{\alpha}_1-\overline{\gamma}_1\right)\right)\nonumber\\
	P_{1,2}&=&r_1\left(\cos\left(\overline{\alpha}_1+\overline{\gamma}_1\right),\sin\left(\overline{\alpha}_1+\overline{\gamma}_1\right)\right)
	\end{eqnarray}
	and
	\begin{eqnarray}\label{steady_config_head2}
	P_{2,1}&=&r_2\left(\cos\left(\overline{\alpha}_2-\overline{\gamma}_2\right),\sin\left(\overline{\alpha}_2-\overline{\gamma}_2\right)\right)\nonumber\\
	P_{2,2}&=&r_2\left(\cos\left(\overline{\alpha}_2+\overline{\gamma}_2\right),\sin\left(\overline{\alpha}_2+\overline{\gamma}_2\right)\right).
	\end{eqnarray}
	\item $\inf G>0$, i.e. the imbalance of the rotor is too large to be compensated by the available balancing masses. Despite that, $(\overline{\alpha}_1,\overline{\gamma}_1;\overline{\alpha}_2,\overline{\gamma}_2)$  is a minimizer of $G$. Hence, by locating the balancing masses in configuration \eqref{steady_config_head1}-\eqref{steady_config_head2}, we do our best to balance the system, being aware full balance cannot be achieved.
\end{itemize}

In the Proposition below, we illustrate when the circumstance $\inf G=0$ occurs.

\begin{proposition}\label{prop_inf G}
	The imbalance indicator $G$ admits zeros $\left(\inf G=0\right)$ if and only if
	\begin{equation}
	m_1r_1\geq \frac{\sqrt{F_{1,x}^2+F_{1,y}^2}}{2\omega^2}\hspace{0.3 cm}\mbox{and}\hspace{0.3 cm}m_2r_2\geq \frac{\sqrt{F_{2,x}^2+F_{2,y}^2}}{2\omega^2}.
	\end{equation}
\end{proposition}
\begin{proof}[Proof of Proposition \ref{prop_inf G}]
	We have $G(\alpha_1,\gamma_1,\alpha_2,\gamma_2)=0$ if and only if
	\begin{equation}
	\begin{dcases}
	2m_1r_1\omega^2\cos(\gamma_1)\cos(\alpha_1)&=F_{1,x}\\
	2m_1r_1\omega^2\cos(\gamma_1)\sin(\alpha_1)&=F_{1,y}\\
	\end{dcases}
	\end{equation}
	\begin{equation}
	\begin{dcases}
	2m_2r_2\omega^2\cos(\gamma_2)\cos(\alpha_2)&=F_{2,x}\\
	2m_2r_2\omega^2\cos(\gamma_2)\sin(\alpha_2)&=F_{2,y}.
	\end{dcases}
	\end{equation}
	Note that the first two equations are decoupled with respect to the second ones. By Lemma \ref{lemma_argming} (3), the above system admits a solution if and only if
	\begin{equation}
	\begin{dcases}
	m_1r_1&\geq \frac{\sqrt{F_{1,x}^2+F_{1,y}^2}}{2\omega^2}\\
	m_2r_2&\geq \frac{\sqrt{F_{2,x}^2+F_{2,y}^2}}{2\omega^2},
	\end{dcases}
	\end{equation}
	as required.
\end{proof}

As we have seen at the beginning of section \ref{sec:2}, an initial configuration $\Phi_0=(\alpha_{0,1},\gamma_{0,1};\alpha_{0,2},\gamma_{0,2})$ of the balancing masses is given. A key issue is to determine a trajectory $\Phi(t)=(\alpha_{0,1}(t),\gamma_{0,1}(t);\alpha_{0,2}(t),\gamma_{0,2}(t))$ joining the initial configuration $\Phi_0$ with a steady optimum $\overline{\Phi}\in\mathscr{S}$ minimizing the imbalance in the meanwhile. For this reason, the \textit{dynamical} control problem has to be addressed. Our main result Proposition \ref{prop_group} asserts the \textit{steady} problem and the \textit{dynamical} one are interlinked.

\subsection{General results}
\label{subsec:2.3}
The purpose of this section is to provide some general tools to prove Proposition \ref{prop_group}. We introduce a generalized version of our functional \eqref{functional}.

Consider the Lagrangian $L:\mathbb{T}^n\times \mathbb{R}^n\longrightarrow \mathbb{R}$
\begin{equation}
L\left(\Phi,\psi\right)\coloneqq \frac{1}{2}\|\psi\|^2+Q(\Phi),
\end{equation}
where $Q:\mathbb{T}^n\longrightarrow \mathbb{R}^+$ is real analytic.

Let $\Phi_0\in \mathbb{T}^n$ be an initial condition. Set the space of admissible trajectories
\begin{equation}\label{adm_traj}
\mathscr{A}\coloneqq \left\{\Phi\in H^1_{loc}([0,+\infty);\mathbb{T}^n) \hspace{0.3 cm} \big| \hspace{0.3 cm} \Phi(0)=\Phi_0\hspace{0.3 cm}\mbox{and}\hspace{0.3 cm}L(\Phi,\dot{\Phi})\in L^1(0,+\infty)\right\}.
\end{equation}
The zero set of $Q$ is denoted by $\mathscr{Z}$.

Our goal is to minimize the functional $
K:\mathscr{A} \longrightarrow \mathbb{R}$
\begin{equation}\label{functional_gen}
K\left(\Phi\right) \coloneqq\int_0^{\infty} \frac{1}{2}\|\dot{\Phi}\|^2+Q(\Phi) \ dt.
\end{equation}

\begin{remark}\label{remark1_gen}
	If $\mathscr{Z}\neq \varnothing$, then the space of admissible trajectories $\mathscr{A}$ is nonempty.
\end{remark}
\begin{proof}
	Take $\overline{\Phi}\in \mathscr{Z}$. Consider the trajectory
	\begin{equation}\Phi(t)\coloneqq
	\begin{dcases}
	(1-t)\Phi_0+t\overline{\Phi} \hspace{0.3 cm} &t\in [0,1)\\
	\overline{\Phi} \hspace{0.3 cm} &t\in [1,+\infty).\\
	\end{dcases}
	\end{equation}
	Now, $\Phi\in \mathscr{A}$, thus showing that $\mathscr{A}\neq \varnothing$.
\end{proof}

In Proposition \ref{prop3_gen}, we are concerned with the existence of minimizer of \eqref{functional_gen}. The proof can be found in the Appendix.

\begin{proposition}\label{prop3_gen}
	There exists $\Phi\in \mathscr{A}$ global minimizer of \eqref{functional_gen}.
\end{proposition}

We now derive to optimality conditions for \eqref{functional_gen}. Let $\Phi\in \mathscr{A}$ be an admissible trajectory. We consider directions $v\in C^{\infty}_c((0,+\infty);\mathbb{R}^n)$. We can compute the directional derivative of $K$ at $\Phi$ along the direction $v$, obtaining
\begin{equation}\label{differential_gen}
\langle dK(\Phi),v\rangle=\int_0^{\infty} \dot{\Phi} \dot{v}+\nabla Q(\Phi) vdt.
\end{equation}

From the above computation of the directional derivative and Fermat's theorem, we derive the first order Optimality Conditions.
\begin{proposition}\label{prop_EL_no_final_condition}
	Take $\Phi$ minimizer of \eqref{functional_gen}. Then, we have:
	\begin{enumerate}
		\item $\Phi\in C^{\infty}([0,+\infty);\mathbb{T}^n)$;
		\item the Euler-Lagrange equations are satisfied
		\begin{equation}\label{EL_no_final_condition}
		\begin{dcases}
		\ddot{\Phi}=\nabla Q\left(\Phi\right)\hspace{0.6 cm}  & \hspace{0.10 cm}t\in (0,+\infty)\\
		\Phi(0)=\Phi_0;
		\end{dcases}
		\end{equation}
		\item the energy is conserved, i.e.
		\begin{equation}\label{Energy_conservation}
		E(t)\coloneqq \frac12\|\dot{\Phi}(t)\|^2-Q\left(\Phi(t)\right)\equiv 0.
		\end{equation}
	\end{enumerate}
\end{proposition}

Now, in the spirit of stabilization-turnpike theory (see \cite{porretta2013long,trelat2015turnpike,TGS}), we show that the time-evolution optima converges as $t\to \infty$ to steady optima. As a byproduct, this will allows us to add to \eqref{EL_no_final_condition} the final condition $\dot{\Phi}(t)\underset{t\to +\infty}{\longrightarrow}0$. We start by proving the following Lemma.

\begin{lemma}\label{lemma_init_datum_estimate}
	Assume $\mathscr{Z}\neq \varnothing$ and $Q$ is real analytic. Let $L\coloneqq\max_{\mathbb{T}^n}\left\|\nabla Q\right\|$ be its Lipschitz constant. Let $d>0$ and $N>0$ be the constants appearing in the \L ojasiewicz inequality (see, e.g. \cite[Th\'eor\`eme 2 page 62]{lojasiewiczensembles}),
	\begin{equation}\label{Lojasiewicz}
	\left|Q(\Phi)\right|\geq d \hspace{0.03 cm} \mbox{dist}(\Phi,\mathscr{Z})^{N}, \hspace{0.3 cm}\forall \ \Phi\in \mathbb{T}^n.
	\end{equation}
	Let $\Phi_0\in \mathbb{T}^n$ be an initial condition. Consider $\Phi\in \mathscr{A}$ global minimizer of \eqref{functional_gen}. Then, for any $t>0$, we have
	\begin{equation}\label{lemma_init_datum_estimate_eq1}
		\mbox{dist}(\Phi(t),\mathscr{Z})\leq \sqrt[\tilde{N}]{\sigma_1\mbox{dist}(\Phi_0,\mathscr{Z})},
	\end{equation}
	and
	\begin{equation}\label{lemma_init_datum_estimate_eq2}
		\mbox{dist}(\Phi(t),\mathscr{Z})\leq \sqrt[N\tilde{N}]{\sigma_2\frac{\mbox{dist}(\Phi_0,\mathscr{Z})}{t}},
	\end{equation}
	where $\tilde{N}\coloneqq \max\left\{2,N\right\}$,
	\begin{equation}
		\sigma_1 \coloneqq \frac{\left(2\pi\sqrt{n}+L\right)}{2}\max\left\{2^{\tilde{N}+1}\left(2\pi\sqrt{n}\right)^{\tilde{N}-2},\frac{2^{\tilde{N}}\left(2\pi\sqrt{n}\right)^{\tilde{N}-N}}{d}\right\}
	\end{equation}
	and
	\begin{equation}
		\sigma_2 \coloneqq \frac{\left(2\pi\sqrt{n}+L\right)\sigma_1^{N}}{2d}.
	\end{equation}
\end{lemma}
One of the consequences of \eqref{lemma_init_datum_estimate_eq2} is the validity of stabilization/turnpike problem for our control problem \eqref{functional_gen}. Indeed, the right hand-side of \eqref{lemma_init_datum_estimate_eq2} decays to zero as $1/t^{N\tilde{N}}$. For this polynomial decay, no assumption are required on the hessian of $Q$.
\begin{proof}[Proof of Lemma \ref{lemma_init_datum_estimate}]
	\textit{Step 1} \ \textbf{Upper bound of $K\left(\Phi\right)=\inf_{\mathscr{A}}K$}\\
	Since $\mathbb{T}^n$ is compact and $Q$ is continuous, the zero set $\mathscr{Z}$ is compact as well, whence by Weierstrass Theorem, there exists $\overline{\Phi}_0\in \mathscr{Z}$, such that $\left\|\Phi_0-\overline{\Phi}_0\right\|=\mbox{dist}(\Phi_0,\mathscr{Z})$. Consider the trajectory
	\begin{equation}\widehat{\Phi}(t)\coloneqq
	\begin{dcases}
	(1-t)\Phi_0+t\overline{\Phi}_0 \hspace{0.3 cm} &t\in [0,1)\\
	\overline{\Phi}_0 \hspace{0.3 cm} &t\in [1,+\infty).\\
	\end{dcases}
	\end{equation}
	Now, on the one hand, for any $t\in [0,1]$
	\begin{eqnarray*}
	Q\left(\widehat{\Phi}(t)\right)&=&\left|Q\left(\widehat{\Phi}(t)\right)-Q\left(\overline{\Phi}_0\right)\right|\nonumber\\
	&\leq &L\left\|(1-t)\Phi_0+t\overline{\Phi}_0-\overline{\Phi}_0\right\|\nonumber\\
	&=&L\left|1-t\right|\left\|\Phi_0-\overline{\Phi}_0\right\|\nonumber\\
	&=&L\left|1-t\right|\mbox{dist}(\Phi_0,\mathscr{Z}).
	\end{eqnarray*}
	On the other hand, for any $t>1$, $\widehat{\Phi}(t)=\overline{\Phi}_0$, whence
	\begin{equation}
		Q\left(\widehat{\Phi}(t)\right) = 0.
	\end{equation}
	Hence,
	\begin{equation}
	\int_0^{+\infty}Q\left(\widehat{\Phi}(t)\right) \ dt\leq \int_0^{1}L\left|1-t\right|\mbox{dist}(\Phi_0,\mathscr{Z}) \ dt=\frac{L}{2}\mbox{dist}(\Phi_0,\mathscr{Z}).
	\end{equation}
	Therefore,
	\begin{align}\label{lemma_init_datum_upper_bound}
		K\left(\Phi\right)&\leq K\left(\widehat{\Phi}\right)\leq \frac12 \left\|\overline{\Phi}_0-\Phi_0\right\|^2+\frac{L}{2}\mbox{dist}(\Phi_0,\mathscr{Z})\nonumber\\
		&\leq\frac{2\pi\sqrt{n}+L}{2}\mbox{dist}(\Phi_0,\mathscr{Z}).\nonumber\\
	\end{align}
	\textit{Step 2} \ \textbf{Lower bound of $K\left(\Phi\right)=\inf_{\mathscr{A}}K$}\\
	Arbitrarily fix $t\geq 0$. We are going to bound $K\left(\Phi\right)$ from below in terms of $\mbox{dist}(\Phi\left(t\right),\mathscr{Z})$, the distance of $\Phi\left(t\right)$ from the zero set $\mathscr{Z}$. Let $s_t\in \mbox{argmin}_{[t,t+1]}Q(\Phi(\cdot))$ and set $\Phi_1\coloneqq \Phi\left(s_t\right)$.
	\begin{align}
		K\left(\Phi\right)&=\int_0^{\infty} \frac{1}{2}\|\dot{\Phi}\|^2+Q(\Phi) \ ds\nonumber\\
		&\geq\int_{t}^{t+1} \frac{1}{2}\|\dot{\Phi}\|^2+Q(\Phi) \ ds\nonumber\\
		&\geq\int_{t}^{t+1} \frac{1}{2}\|\dot{\Phi}\|^2+\min_{\xi\in [t,t+1]} Q\left(\Phi(\xi)\right) \ ds\nonumber\\
		&=\int_{t}^{t+1} \frac{1}{2}\|\dot{\Phi}\|^2+Q(\Phi_1) \ ds\nonumber\\
		&=\int_{t}^{t+1} \frac{1}{2}\|\dot{\Phi}\|^2 \ ds + Q(\Phi_1)\nonumber\\
		&\geq\int_{t}^{s_t} \frac{1}{2}\|\dot{\Phi}\|^2 \ ds + Q(\Phi_1)\nonumber\\
		&\geq\inf_{\Gamma\in \mathscr{A}_{\Phi(t),\Phi_1}}\int_{t}^{s_t} \frac{1}{2}\|\dot{\Gamma}\|^2 \ ds + Q(\Phi_1)\label{lemma_init_datum_estimate_eq_38_3}\\
		&\geq\frac{\left\|\Phi_1-\Phi(t)\right\|^2}{2\left(s_t-t\right)} + Q(\Phi_1)\nonumber\\
		&\geq\frac{\left\|\Phi_1-\Phi(t)\right\|^2}{2} + Q(\Phi_1)\label{lemma_init_datum_estimate_eq_38_6},\\
	\end{align}
	where \eqref{lemma_init_datum_estimate_eq_38_6} is justified by $\left|s_t-t\right|\leq 1$ and in \eqref{lemma_init_datum_estimate_eq_38_3} we minimize over the space of trajectories $\Gamma$ linking $\Phi(t)$ and $\Phi_1$ in time $s_t-t$
	\begin{equation}
	\mathscr{A}_{\Phi(t),\Phi_1}\coloneqq \left\{\Gamma\in H^1\left(\left(t,s_t\right);\mathbb{T}^n\right) \hspace{0.3 cm} \big| \hspace{0.3 cm} \Gamma(t)=\Phi(t)\hspace{0.3 cm}\mbox{and}\hspace{0.3 cm}\Gamma\left(s_t\right)=\Phi_1\right\}.
	\end{equation}
	We now employ the above inequality combined with \eqref{Lojasiewicz}, getting
	\begin{eqnarray}\label{lemma_init_datum_estimate_eq_49}
		K\left(\Phi\right)&\geq&\frac{\left\|\Phi_1-\Phi(t)\right\|^2}{2} + Q(\Phi_1)\nonumber\\
		&\geq&\frac{\left\|\Phi_1-\Phi(t)\right\|^2}{2}+d \hspace{0.03 cm} \mbox{dist}(Q(\Phi_1),\mathscr{Z})^{N}\nonumber\\
		&=&\frac{\left\|\Phi_1-\Phi(t)\right\|^2}{2}+d \hspace{0.03 cm} \left\|\overline{\Phi}_1-\Phi_1\right\|^{N}\nonumber\\
		&\geq&\frac{\left\|\Phi_1-\Phi(t)\right\|^{\tilde{N}}}{2\left(2\pi\sqrt{n}\right)^{\tilde{N}-2}}+\frac{d}{\left(2\pi\sqrt{n}\right)^{\tilde{N}-N}}  \left\|\overline{\Phi}_1-\Phi_1\right\|^{\tilde{N}}\nonumber\\
		&\geq&\frac{1}{\max\left\{2\left(2\pi\sqrt{n}\right)^{\tilde{N}-2},\frac{\left(2\pi\sqrt{n}\right)^{\tilde{N}-N}}{d} \right\}}\left[\left\|\Phi_1-\Phi(t)\right\|^{\tilde{N}}+\left\|\overline{\Phi}_1-\Phi_1\right\|^{\tilde{N}}\right]\nonumber\\
		&\geq&\zeta^{-1}\left[\left\|\Phi_1-\Phi(t)\right\|+\left\|\overline{\Phi}_1-\Phi_1\right\|\right]^{\tilde{N}}\nonumber\\
		&\geq&\zeta^{-1}\left\|\Phi(t)-\overline{\Phi}_1\right\|^{\tilde{N}}\geq\zeta^{-1}\mbox{dist}(\Phi(t),\mathscr{Z})^{\tilde{N}},
	\end{eqnarray}
	with $\overline{\Phi}_1\in \mbox{argmin}_{\mathscr{Z}}\left\|\cdot-\Phi_1\right\|$, $\tilde{N}= \max\left\{2,N\right\}$ and
	\begin{equation}
	\zeta \coloneqq \max\left\{2^{\tilde{N}+1}\left(2\pi\sqrt{n}\right)^{\tilde{N}-2},\frac{2^{\tilde{N}}\left(2\pi\sqrt{n}\right)^{\tilde{N}-N}}{d}\right\}.
	\end{equation}
	
	\textit{Step 3} \ \textbf{Proof of \eqref{lemma_init_datum_estimate_eq1}}\\
	By using \eqref{lemma_init_datum_upper_bound} and \eqref{lemma_init_datum_estimate_eq_49}, we obtain
	\begin{eqnarray}\label{lemma_init_datum_estimate_eq_52}
		\frac{2\pi\sqrt{n}+L}{2}\mbox{dist}(\Phi_0,\mathscr{Z})&\geq&K\left(\Phi\right)\nonumber\\
		&\geq&\zeta^{-1}\mbox{dist}(\Phi(t),\mathscr{Z})^{\tilde{N}},
	\end{eqnarray}
	as required. We now aim at proving \eqref{lemma_init_datum_estimate_eq2}.
	
	\textit{Step 4} \ \textbf{Bound of $\min_{\tau\in [0,t]}\mbox{dist}(\Phi(\tau),\mathscr{Z})$}\\
	Arbitrarily fix $t>0$. We have
	\begin{align}
	K\left(\Phi\right)&=\int_0^{\infty} \frac{1}{2}\|\dot{\Phi}\|^2+Q(\Phi) \ ds\nonumber\\
	&\geq\int_{0}^{t} Q(\Phi) \ ds\nonumber\\
	&\geq\int_{0}^{t} d \hspace{0.03 cm} \mbox{dist}(\Phi,\mathscr{Z})^{N} \ ds\nonumber\\
	&\geq d\int_{0}^{t} \min_{\tau\in [0,t]}\mbox{dist}(\Phi(\tau),\mathscr{Z})^{N} \ ds\nonumber\\
	&= td \min_{\tau\in [0,t]}\mbox{dist}(\Phi(\tau),\mathscr{Z})^{N}\label{lemma_init_datum_estimate_time_contraction_estimate_eq_6_6}.\\
	\end{align}
	Therefore, by \eqref{lemma_init_datum_upper_bound},
	\begin{equation}
		td \min_{\tau\in [0,t]}\mbox{dist}(\Phi(\tau),\mathscr{Z})^{N}\leq K\left(\Phi\right)\leq \frac{2\pi\sqrt{n}+L}{2}\mbox{dist}(\Phi_0,\mathscr{Z}),
	\end{equation}
	whence
	\begin{equation}\label{lemma_init_datum_estimate_min_bound}
		\min_{\tau\in [0,t]}\mbox{dist}(\Phi\left(\tau\right),\mathscr{Z})\leq \sqrt[N]{\frac{2\pi\sqrt{n}+L}{2td}\mbox{dist}(\Phi_0,\mathscr{Z})}.
	\end{equation}
	\textit{Step 5} \ \textbf{Proof of \eqref{lemma_init_datum_estimate_eq2}}\\
	Let $\tau_t\in \mbox{argmin}_{[0,t]}\mbox{dist}(\Phi(\cdot),\mathscr{Z})$. By applying \eqref{lemma_init_datum_estimate_eq1} with initial datum $\Phi\left(\tau_t\right)$ at initial time $\tau_t$, we get
	\begin{align}
	\mbox{dist}(\Phi(t),\mathscr{Z})&\leq\sqrt[\tilde{N}]{\sigma_1\mbox{dist}\left(\Phi\left(\tau_t\right),\mathscr{Z}\right)}\nonumber\\
	&=\sqrt[\tilde{N}]{\sigma_1\min_{\tau\in [0,t]}\mbox{dist}(\Phi(\tau),\mathscr{Z})}.\\
	\end{align}
	Then, combining the above inequality with \eqref{lemma_init_datum_estimate_min_bound}, we obtain
	\begin{align}
		\mbox{dist}(\Phi(t),\mathscr{Z})&\leq\sqrt[\tilde{N}]{\sigma_1\min_{\tau\in [0,t]}\mbox{dist}(\Phi(\tau),\mathscr{Z})}\nonumber\\
		&\leq\sqrt[\tilde{N}]{\sigma_1\sqrt[N]{\frac{2\pi\sqrt{n}+L}{2td}\mbox{dist}(\Phi_0,\mathscr{Z})}}\nonumber\\
		&=\sqrt[N\tilde{N}]{\frac{\left(2\pi\sqrt{n}+L\right)\sigma_1^N}{2d}\frac{\mbox{dist}(\Phi_0,\mathscr{Z})}{t}},\\
	\end{align}
	as required.
\end{proof}

\begin{remark}\label{remark_saturation}
	As kindly suggested by one of the referees, in practical applications, the control input $\dot{\Phi}$ is subject to saturation. Indeed, physical constraints must be taken into account, e.g. balancing masses cannot rotate too fast. Several rotors control strategies available in the literature consider saturation effects, such as \cite{FU201852,ZHOU202011153}. In our case, we are able to guarantee that our control magnitude does not exceed an explicit threshold, constrains being intrinsically imposed in the functional definition \eqref{functional}. Indeed, let us work in the framework of Lemma \ref{lemma_init_datum_estimate}. By employing \eqref{lemma_init_datum_estimate_eq1} and energy conservation \eqref{Energy_conservation}, we obtain
	\begin{equation}
		\frac12 \left\|\dot{\Phi}(t)\right\|^2=Q\left(\Phi(t)\right)\leq L\mbox{dist}(\Phi(t),\mathscr{Z})\leq L\sqrt[\tilde{N}]{\sigma_1\mbox{dist}(\Phi_0,\mathscr{Z})},
	\end{equation}
	whence
	\begin{equation}
		\left\|\dot{\Phi}(t)\right\|\leq \sqrt{2L}\sqrt[2\tilde{N}]{\sigma_1\mbox{dist}(\Phi_0,\mathscr{Z})},
	\end{equation}
	for any $t\geq 0$. This gives an upper bound for the magnitude of the optimal control.
\end{remark}

\begin{proposition}\label{prop4gen}
	Assume $\mathscr{Z}\subset \mathbb{T}^n$ is nonempty and finite and $Q$ real analytic. Consider $\Phi\in \mathscr{A}$ global minimizer of \eqref{functional_gen}. Then,
	\begin{enumerate}
		\item there exists $\overline{\Phi}\in \mbox{argmin}\left(Q\right)$ such that
		\begin{equation}\label{Zfinite_state_turnpike}
		\Phi(t)\underset{t\to +\infty}{\longrightarrow}\overline{\Phi},
		\end{equation}
		\begin{equation}\label{velocity_conv}
		\dot{\Phi}(t)\underset{t\to +\infty}{\longrightarrow}0
		\end{equation}
		and
		\begin{equation}
		\left|Q\left(\Phi(t)\right)\right|\underset{t\to +\infty}{\longrightarrow}0.
		\end{equation}
		\item the Euler-Lagrange equations can be complemented with final condition
		\begin{equation}\label{EL_gen}
		\begin{dcases}
		\ddot{\Phi}=\nabla Q\left(\Phi\right)\hspace{0.6 cm}  & \hspace{0.10 cm}t\in (0,+\infty)\\
		\Phi(0)=\Phi_0, \hspace{0.16 cm} \dot{\Phi}(T)\underset{T\to +\infty}{\longrightarrow}0.
		\end{dcases}
		\end{equation}		
	\end{enumerate}
\end{proposition}

\begin{proof}[Proof of Proposition \ref{prop4gen}]
	We start proving \eqref{Zfinite_state_turnpike}.\\
	Let $\Phi$ be a minimizer of \eqref{functional_gen}. Let us prove \eqref{Zfinite_state_turnpike}. prove Since $\mathscr{Z}\subset \mathbb{T}^n$ is finite and $\Phi$ is continuous,
	there exists a unique $\overline{\Phi}\in\mathscr{Z}$, such that
	\begin{equation}
	\mbox{dist}(\Phi(t),\mathscr{Z})=\|\Phi(t)-\overline{\Phi}\|,\hspace{0.3 cm}\forall \ t\geq \overline{t},
	\end{equation}
	with $\overline{t}$ large enough. By the above equality and \eqref{lemma_init_datum_estimate_eq2}, we have
	\begin{equation}
	\|\Phi(t)-\overline{\Phi}\|= \mbox{dist}(\Phi(t),\mathscr{Z})\underset{t\to +\infty}{\longrightarrow}0,
	\end{equation}
	as required. By the above convergence and \eqref{Energy_conservation}, we immediately have $\dot{\Phi}(t)\underset{t\to +\infty}{\longrightarrow}0$.
	
	\eqref{EL_gen} follows from Proposition \ref{prop_EL_no_final_condition} together with \eqref{velocity_conv}.
\end{proof}

Note that \eqref{EL_gen} can be seen as a system of two coupled elliptic PDEs, with a Dirichlet condition at time $t=0$ and a Neumann condition at $t=+\infty$.

Equivalently, we can formulate the first order optimality conditions as a state-adjoint state first order system.
\begin{equation}\label{Pontryagin_gen}
\begin{dcases}
\dot{\Phi}=-q\hspace{0.6 cm} & \hspace{0.10 cm}t\in (0,+\infty)\\
-\dot{q}=\nabla Q(\Phi)\hspace{0.6 cm} & \hspace{0.10 cm}t\in (0,+\infty)\\
\Phi(0)=\Phi_0, \hspace{0.16 cm} q(T)\underset{T\to +\infty}{\longrightarrow}0.
\end{dcases}
\end{equation}

By using the above optimality system, we can improve the decay rate estimate.
\begin{proposition}\label{prop_stab_gen_exp}
	Suppose $\mathscr{Z}\subset \mathbb{T}^n$ is nonempty and finite and $Q$ real analytic. In addition, assume
	\begin{equation}\label{Hessian_condition}
	\nabla^2Q(\overline{\Phi})\hspace{0.3 cm}\mbox{is (strictly) positive definite},
	\end{equation}
	Consider $\Phi\in \mathscr{A}$ global minimizer of \eqref{functional_gen}. Then, we have the exponential estimate, for any $t\geq 0$
	\begin{equation}\label{prop4gen_eq8}
	\|\Phi(t)-\overline{\Phi}\|+\|\dot{\Phi}(t)\|+\left|Q\left(\Phi(t)\right)\right|\leq C\exp\left(-\mu t\right),
	\end{equation}
	with $C, \hspace{0.07 cm}\mu >0$ independent of $t$.
\end{proposition}
The proof of Proposition \ref{prop_stab_gen_exp} is a consequence of Proposition \ref{prop4gen} and Lemma \ref{lemma_exp_est_gen} stated and proved below, inspired by \cite{trelat2015turnpike} and \cite{TGS}.
\begin{lemma}\label{lemma_exp_est_gen}
	Let $\Phi_0\in\mathbb{T}^n$ and $\Phi \in C^{\infty}(\mathbb{R}^+;\mathbb{T}^n)$ solution to
	\begin{equation}\label{EL_5_gen}
	\begin{dcases}
	\ddot{\Phi}=\nabla Q(\Phi)\hspace{0.6 cm} & \hspace{0.10 cm}t\in (0,+\infty)\\
	\Phi(0)=\Phi_0.
	\end{dcases}
	\end{equation}
	Suppose the existence $\overline{\Phi}\in\mathscr{Z}$ such that
	\begin{equation}\label{conv_1_gen}
	\Phi(t)\underset{t\to +\infty}{\longrightarrow}\overline{\Phi}
	\end{equation}
	and
	\begin{equation}\label{conv_2_gen}
	\dot{\Phi}(t)\underset{t\to +\infty}{\longrightarrow}0.
	\end{equation}
	Assume the condition \eqref{Hessian_condition} holds. Then,
	\begin{equation}
	\|\Phi(t)-\overline{\Phi}\|+\|\dot{\Phi}(t)\|+\left|Q\left(\Phi(t)\right)\right|\leq C\exp\left(-\mu t\right),\hspace{0.3 cm}\forall \ t\geq 0,
	\end{equation}
	with $\mu >0$.
\end{lemma}
\begin{proof}[Proof of Lemma \ref{lemma_exp_est_gen}]
	\textit{Step 1} \  \textbf{Reduction the a first order problem}\\
	Take any $\Phi$ solution to \eqref{EL_5_gen}. Then, the function
	\begin{equation}
	\textbf{x}\coloneqq 
	\begin{tikzpicture}[baseline=(current bounding box.center)]
	\matrix (b) [matrix of math nodes,nodes in empty cells,right delimiter={]},left delimiter={[} ]{
		\Phi-\overline{\Phi} \\
		\dot{\Phi}  \\
	} ;
	\end{tikzpicture}
	\end{equation}
	solves the first order problem
	\begin{equation}\label{EL0_first_order}
	\begin{dcases}
	\dot{\textbf{x}}= \textbf{f}( \textbf{x})\hspace{0.6 cm} & \hspace{0.10 cm}t\in (0,+\infty)\\
	\textbf{x}(T)\underset{T\to +\infty}{\longrightarrow}0.
	\end{dcases}
	\end{equation}
	where
	\begin{equation}
	\textbf{f}( \textbf{x})\coloneqq
	\begin{tikzpicture}[baseline=(current bounding box.center)]
	\matrix (b) [matrix of math nodes,nodes in empty cells,right delimiter={]},left delimiter={[} ]{
		x_{n+1}  \\
		\\
		\\
		\\
		\\
		x_{2n}  \\
		\nabla Q\left(\left(x_1,\dots,x_n\right)+\overline{\Phi}\right) \\
	} ;
	\draw[loosely dotted] (b-2-1)--(b-6-1);
	\end{tikzpicture}.
	\end{equation}
	\textit{Step 2} \  \textbf{$\textbf{0}$ is an hyperbolic equilibrium point.}\\
	We observe that $\textbf{f}(0)=0$, since $\overline{\Phi}$ is a zero of $Q$. Moreover, the Jacobian of $ \textbf{f}$ at $ \textbf{x}=0$ is a block matrix
	\begin{equation}
	D\textbf{f}(0)=\begin{pmatrix}
	0&I_n\\
	\nabla^2Q\left(\overline{\Phi}\right)&0,
	\end{pmatrix}
	\end{equation}
	where $I_n$ is the $n\times n$ identity matrix. By assumption \eqref{Hessian_condition}, $\nabla^2Q\left(\overline{\Phi}\right)$ is positive definite. Then, there exists $C$ symmetric positive definite, such that $C^2=\nabla^2Q\left(\overline{\Phi}\right)$. Following \cite[subsection III.B]{SMA}, we introduce the matrix
	\begin{equation}\label{Lambda}
	\Lambda\coloneqq
	\frac{1}{2}\begin{pmatrix}
	2I_n&-C^{-1}\\
	2C&I_n
	\end{pmatrix}.
	\end{equation}
	Since $\nabla^2Q\left(\overline{\Phi}\right)$ is (strictly) positive definite, $\Lambda$ is invertible\footnote{
		\begin{equation}
		\Lambda^{-1}=\begin{pmatrix}
		\frac12I_2&\frac12 C^{-1}\\
		-C&I_2.
		\end{pmatrix}
		\end{equation}
	} and
	\begin{equation}
	\Lambda^{-1}D\textbf{f}(0) \hspace{0.1 cm} \Lambda=\begin{pmatrix}
	C&0\\
	0&-C.
	\end{pmatrix}
	\end{equation}
	Hence, the spectrum of the jacobian $D\textbf{f}(0)$ does not intersect the imaginary axis, whence $\textbf{0}$ is an hyperbolic equilibrium point for \eqref{EL0_first_order}, as required.
	\\
	\textit{Step 3} \  \textbf{Conclusion by applying the Stable Manifold Theorem}\\
	As we have seen in step 2, $\textbf{0}$ is an hyperbolic equilibrium point for \eqref{EL0_first_order}. Then, by the Stable Manifold Theorem (see e.g. \cite[section 2.7]{perko2013differential} or \cite{TGS}), the stable and unstable manifolds for \eqref{EL0_first_order} exist in a neighborhood of $\textbf{0}$. Besides, thanks to \eqref{conv_1_gen} and \eqref{conv_2_gen}, $\textbf{x}=\left(\Phi-\overline{\Phi},\dot{\Phi}\right)$ belongs to the stable manifold of the above problem.
	
	Hence, by Stable Manifold theory (see, e.g. \cite[Corollary page 115]{perko2013differential} or \cite{TGS}), we have for some $\mu >0$
	\begin{equation}\label{exp_est_parz_1}
	\| \textbf{x}(t)\|\leq C\exp(-\mu t),\hspace{0.3 cm}\forall \ t\geq 0,
	\end{equation}
	which yields
	\begin{equation}
	\|\Phi(t)-\overline{\Phi}\|+\|\dot{\Phi}(t)\|\leq C\exp\left(-\mu t\right),\hspace{0.3 cm}\forall \ t\geq 0.
	\end{equation}		
	To conclude the proof, we observe that $Q$ is globally Lipschitz
	and $Q\left(\overline{\Phi}\right)=0$. Then,
	\begin{eqnarray*}
		\left|Q\left(\Phi(t)\right)\right|&=&\left|Q\left(\Phi(t)\right)-Q\left(\overline{\Phi}\right)\right|\leq L\|\Phi(t)-\overline{\Phi}\|\nonumber\\
		&\leq& C\exp\left(-\mu t\right).
	\end{eqnarray*}
	where in the last inequality we have employed \eqref{exp_est_parz_1}.
\end{proof}

\subsection{Proof of Proposition \ref{prop_group}}
\label{subsec:2.4}

We prove Proposition \ref{prop_group} employing the general results of subsection \ref{subsec:2.3}. The numbering $(1)$, $(2)$ and $(3)$ refers to the numbered statements in Proposition \ref{prop_group}.

\begin{proof}[Proof of Proposition \ref{prop_group}]
	The existence of minimizers for \eqref{functional} follows from Proposition \ref{prop3_gen}, with $K= J$.\\
	\textit{Step 1} \ \textbf{Reduction to two angles}\\
	By \eqref{split}, the imbalance indicator  splits as $G(\alpha_1,\gamma_1,\alpha_2,\gamma_2)=G_1(\alpha_1,\gamma_1)+G_2(\alpha_2,\gamma_2)$, whence
	\begin{equation}
	\hat{G}(\alpha_1,\gamma_1,\alpha_2,\gamma_2)=\hat{G}_1(\alpha_1,\gamma_1)+\hat{G}_2(\alpha_2,\gamma_2),
	\end{equation}
	with $\hat{G}_1(\alpha_1,\gamma_1)\coloneqq G_1(\alpha_1,\gamma_1)-\inf G_1$ and $\hat{G}_2(\alpha_2,\gamma_2)\coloneqq G_2(\alpha_2,\gamma_2)-\inf G_2$. Then, the functional
	\begin{equation}\label{functional_split}
	J(\Phi)=J_1(\alpha_1,\gamma_1)+J_2(\alpha_2,\gamma_2),
	\end{equation}
	where
	\begin{equation}\label{functional_split_1}
	J_1(\alpha_1,\gamma_1)\coloneqq\frac{1}{2}\int_0^{\infty} \left[|\dot{\alpha}_1|^2+|\dot{\gamma}_1|^2+{\beta}\hat{G}_1(\alpha_1,\gamma_1)\right] dt
	\end{equation}
	and
	\begin{equation}\label{functional_split_2}
	J_2(\alpha_2,\gamma_2)\coloneqq\frac{1}{2}\int_0^{\infty} \left[|\dot{\alpha}_2|^2+|\dot{\gamma}_2|^2+{\beta}\hat{G}_2(\alpha_2,\gamma_2)\right] dt.
	\end{equation}
	This enables us to work on $J_1$ and $J_2$ separately. From the physical viewpoint, the functional $J_1$ is related to the first balancing head, while $J_2$ is related to the second balancing head. Both $J_1$ and $J_2$ fit in a general class of functionals \eqref{functional_gen}, defining
	\begin{eqnarray}
	Q_i(\alpha_i,\gamma_i)&\coloneqq&\frac{\beta}{2}\left[\left|\cos(\gamma_i)\cos(\alpha_i)-c^i_1\right|^2\right. \label{Qi} \\
	\nonumber&\;&\left.+\left|\cos(\gamma_i)\sin(\alpha_i)-c^i_2\right|^2\right],
	\end{eqnarray}
	possibly remaining $\beta$ after the absorption of the coefficient $\frac{1}{2m_ir_i\omega^2}$ and
	\begin{equation}\label{c_J_i_proof}
	c^i=\frac{1}{{2m_ir_i\omega^2}}\left(F_{i,x},F_{i,y}\right).
	\end{equation}
	\textit{Step 2} \  \textbf{Proof of (2)}\\
	For any $\Phi=\left(\alpha_1,\gamma_1;\alpha_2,\gamma_2\right)$ minimizer of \eqref{functional}, $\left(\alpha_1,\gamma_1\right)$ minimizes $J_1$ and $\left(\alpha_2,\gamma_2\right)$ minimizes $J_2$. We apply Proposition \ref{prop_EL_no_final_condition} to $J_1$ and $J_2$, computing the gradient of $Q_i$ defined in \eqref{Qi}
	\begin{eqnarray}\label{nablaQi_EL}
	\frac{\partial Q_i}{\partial \alpha_i}\left(\alpha_i,\gamma_i\right)&=&\beta\cos\left(\gamma_i\right)\left[c^i_1\sin\left(\alpha_i\right)-c^i_2\cos\left(\alpha_i\right)\right]\nonumber\\
	\frac{\partial Q_i}{\partial \gamma_i}\left(\alpha_i,\gamma_i\right)&=&\beta\sin\left(\gamma_i\right)\left[c^i_1\cos\left(\alpha_i\right)+c^i_2\sin\left(\alpha_i\right)\right.\nonumber\\
	&\;&\left.-\cos\left(\gamma_i\right)\right].
	\end{eqnarray}
	\textit{Step 3} \\  \textbf{Proof of (3) and (4)}\\
	By Step 1, we reduce to prove the assertion for minimizers of $J_1$ and $J_2$. Let $\left(\alpha_i,\gamma_i\right)$ be a minimizer of $J_i$, for some $i=1,2$.
	
	\vspace*{4pt}\noindent\textbf{Case 1.} $\mbox{argmin}\left(Q_i\right)\subset \mathbb{T}^2$ is finite.
	
	If $\mbox{argmin}\left(Q_i\right)\subset \mathbb{T}^2$ is finite, we directly apply Proposition \ref{prop4gen} to $K\coloneqq J_i$, getting the required convergences. If, in addition, \eqref{prop_group_eq4} is verified, we want to prove that the Hessian of $Q_i$ at the steady optimum is positive definite. To this end, we compute $\nabla^2Q_i(\alpha_i,\gamma_i)$
	\begin{eqnarray}\label{secondderivativesQ_i_3}
	\frac{\partial^2 Q_i}{\partial \alpha^2}\left(\alpha_i,\gamma_i\right)&=&\beta\cos\left(\gamma_i\right)\left[c^i_1\cos\left(\alpha_i\right)+c^i_2\sin\left(\alpha_i\right)\right]\nonumber\\
	\frac{\partial^2 Q_i}{\partial \gamma_i^2}\left(\alpha_i,\gamma_i\right)&=&\beta\cos\left(\gamma_i\right)\left[c^i_1\cos\left(\alpha_i\right)+c^i_2\sin\left(\alpha_i\right)\right.\nonumber\\
	&\;&\left.-\cos\left(\gamma_i\right)\right]+\beta\sin\left(\gamma_i\right)^2\nonumber\\
	\frac{\partial^2 Q_i}{\partial \gamma_i\partial\alpha_i}\left(\alpha_i,\gamma_i\right)&=&\beta\sin\left(\gamma_i\right)\left[-c^i_1\sin\left(\alpha_i\right)+c^i_2\cos\left(\alpha_i\right)\right]\nonumber\\
	\end{eqnarray}
	Now, let $\overline{\Phi}\in\mbox{argmin}\left(Q_i\right)$.
	Since $\overline{\Phi}\in\mbox{argmin}\left(Q_i\right)$ and \eqref{prop_group_eq4} holds, by Lemma \ref{lemma_argming},
	\begin{equation}
	c^i=\cos\left(\overline{\gamma}_i\right)\left(\cos\left(\overline{\alpha}_i\right),\sin\left(\overline{\alpha}_i\right)\right).
	\end{equation}
	and $\sin\left(\overline{\gamma}_i\right)\neq 0$. Hence, by \eqref{nablaQi_EL}, $c_1\cos\left(\overline{\alpha}_i\right)+c_2\sin\left(\overline{\alpha}_i\right)-\cos\left(\overline{\gamma}_i\right)=0$. We plug these results into \eqref{secondderivativesQ_i_3}, obtaining
	\begin{eqnarray}\label{pos_def_hess}
	\frac{\partial^2 Q_i}{\partial \alpha_i^2}\left(\overline{\alpha}_i,\overline{\gamma}_i\right)&=&\beta\cos\left(\overline{\gamma}_i\right)^2\left[\cos\left(\overline{\alpha}_i\right)^2+\sin\left(\overline{\alpha}_i\right)^2\right]\nonumber\\
	&=&\beta\cos\left(\overline{\gamma}_i\right)^2\nonumber\\
	\frac{\partial^2 Q_i}{\partial \gamma_i^2}\left(\overline{\alpha}_i,\overline{\gamma}_i\right)&=&\beta\sin\left(\overline{\gamma}_i\right)^2\nonumber\\
	\frac{\partial^2 Q_i}{\partial \gamma_i\partial\alpha_i}\left(\overline{\alpha_i},\overline{\gamma}_i\right)&=&\beta\|c\|\sin\left(\overline{\gamma}_i\right)\left[-\cos\left(\overline{\alpha}_i\right)\sin\left(\overline{\alpha}_i\right)\right.\nonumber\\
	&\;&\left.+\sin\left(\overline{\alpha}_i\right)\cos\left(\overline{\alpha}_i\right)\right]=0,
	\end{eqnarray}
	namely the Hessian of $Q_i$ computed at $\left(\overline{\alpha}_i,\overline{\gamma}_i\right)$ is diagonal. Using once more \eqref{prop_group_eq4} and by Lemma \ref{lemma_argming}, we have both $\cos\left(\overline{\gamma}\right)\neq 0$ and $\sin\left(\overline{\gamma}\right)\neq 0$. Then,
	the Hessian of $Q_i$ computed at $\left(\overline{\alpha}_i,\overline{\gamma}_i\right)$ is (strictly) positive definite. We apply Proposition \ref{prop4gen} (2) to conclude.
	
	\vspace*{4pt}\noindent\textbf{Case 2.} $\mbox{argmin}\left(Q_i\right)\subset \mathbb{T}^2$ is a continuum.
	
	From the physical viewpoint, this occurs when in the plane $\pi_i$ there is no imbalance, namely $F_i=0$. Now, by Lemma \ref{lemma_argming}, $\mbox{argmin}\left(Q_i\right)\subset \mathbb{T}^2$ is a continuum if and only if $c^i=0$, namely
	\begin{equation}
	Q_i(\alpha_i,\gamma_i)= \frac{\beta}{2}\left[\left|\cos(\gamma_i)\cos(\alpha_i)\right|^2+\left|\cos(\gamma_i)\sin(\alpha_i)\right|^2\right].
	\end{equation}
	and the Euler-Lagrange equations satisfied by $\left(\alpha_i,\gamma_i\right)$ read as
	\begin{equation}\label{EL_new_3}
	\begin{dcases}
	\ddot{\alpha}_i=0\hspace{0.6 cm}  & \hspace{0.10 cm}t\in (0,+\infty)\\
	\ddot{\gamma}_i=-\frac{\beta}{2}\sin\left(2\gamma\right)\hspace{0.6 cm}  & \hspace{0.10 cm}t\in (0,+\infty)\\
	\alpha(0)=\alpha_0, \hspace{0.16 cm} \dot{\alpha}(T)\underset{T\to +\infty}{\longrightarrow}0\\
	\gamma(0)=\gamma_0, \hspace{0.16 cm} \dot{\gamma}(T)\underset{T\to +\infty}{\longrightarrow}0.
	\end{dcases}
	\end{equation}
	This entails that
	\begin{equation}\label{alpha_const}
	\alpha(t)\equiv \alpha_0.
	\end{equation}
	Furthermore, for any integer $k$, $\cos((2k+1)\pi)<0$.
	Therefore, we are in position to conclude applying Proposition \ref{prop4gen} to the functional
	\begin{equation}
	K\left(\gamma_i\right)\coloneqq \frac{1}{2}\int_0^{\infty} \left[\|\dot{\gamma}_i\|^2+\beta\left|\cos\left(\gamma\right)\right|^2\right] dt.
	\end{equation}
	In case $\mbox{argmin}\left(Q_i\right)\subset \mathbb{T}^2$ is a continuum, the above proof can be seen from the point of view of phase analysis. Indeed, the Euler-Lagrange equations reduce to the pendulum-like equation
	\begin{equation}\label{EL_new_3_gamma}
	\begin{dcases}
	\ddot{\gamma}=-\frac{\beta}{2}\sin\left(2\gamma\right)\hspace{0.6 cm}  & \hspace{0.10 cm}t\in (0,+\infty)\\
	\gamma(0)=\gamma_0, \hspace{0.16 cm} \dot{\gamma}(T)\underset{T\to +\infty}{\longrightarrow}0.
	\end{dcases}
	\end{equation}
	\begin{figure}[htp]
		\begin{center}
			\includegraphics[width=11cm]{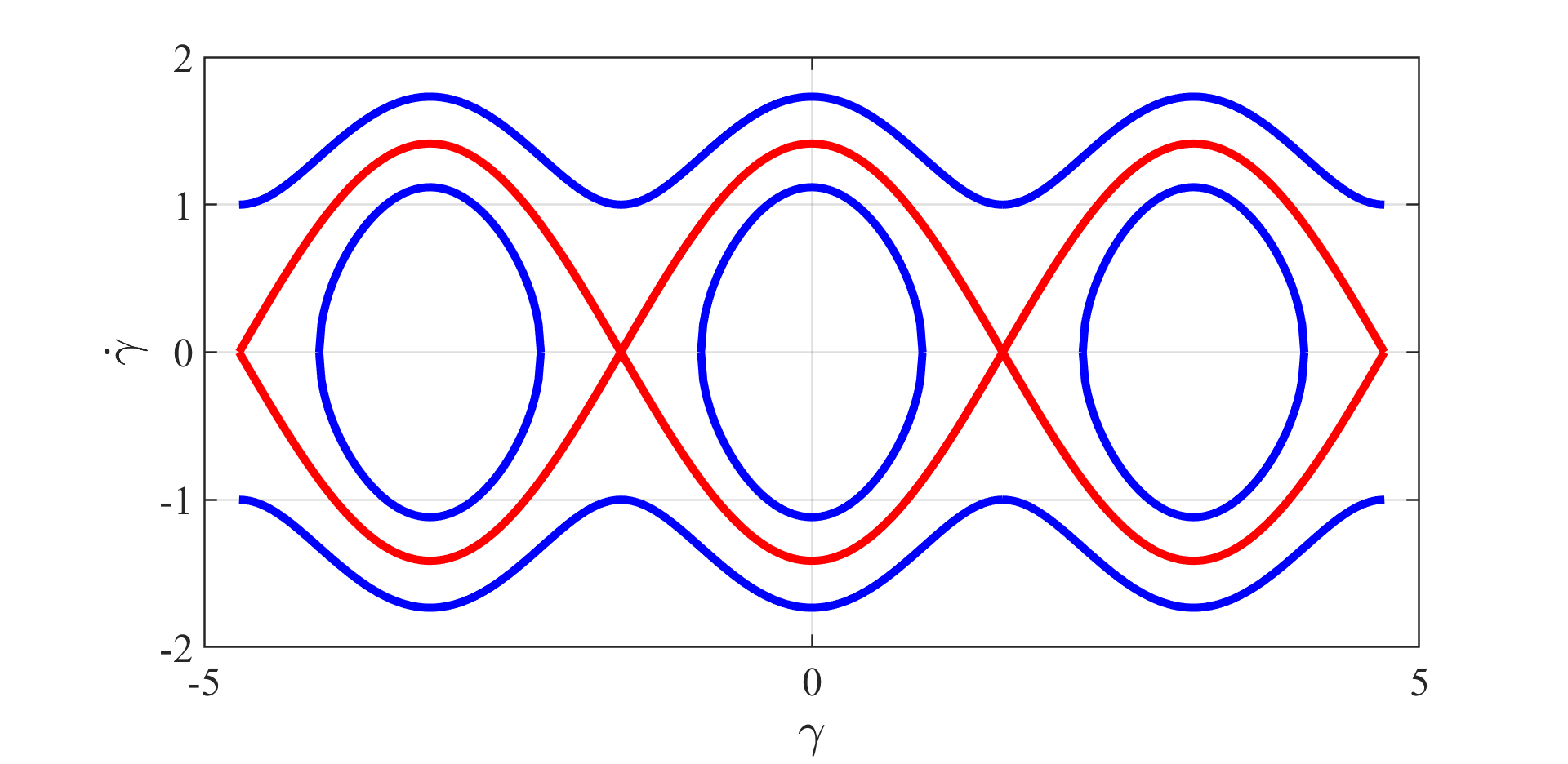}\\
			\caption{Phase portrait for the Euler-Lagrange equations in the balanced case. The red curve is the separatrix.}
			\label{phaseportrait1}
		\end{center}
	\end{figure}
	We have the end condition $\dot{\gamma}(T)\underset{T\to +\infty}{\longrightarrow}0$. Then, any solution $\gamma$ of \eqref{EL_new_3_gamma} lies on the separatrix (the red curve in figure \ref{phaseportrait1}), so that it must stabilize towards some steady state.
\end{proof}

\subsection{Numerical simulations}
\label{subsec:2.5}

In order to perform some numerical simulations, we firstly discretize our functional \eqref{functional_gen} and then we run \verb!AMPL!-\verb!IPOpt! to minimize the resulting discretized functional.

For the purpose of the numerical simulations, it is convenient to rewrite \eqref{functional_gen} as
\begin{equation}
\widetilde{K}\left(\psi,\Phi\right) \coloneqq \int_0^{\infty}\frac12 \|\dot{\Phi}\|^2+Q\left(\Phi\right)dt,
\end{equation}
subject to the state equation
\begin{equation}\label{}
\begin{dcases}
\frac{d}{dt}\Phi=\psi\hspace{0.6 cm} &t\in (0,+\infty)\\
\Phi(0)=\Phi_0.\\
\end{dcases}
\end{equation}

\subsubsection{Discretization}
\label{subsubsec:2.5.1}
Let $\varepsilon >0$ and suppose we want to get $\varepsilon$-close to the stable configuration at some time $T>0$, i.e. $\mbox{dist}(\Phi\left(T\right),\mathscr{Z})<\varepsilon$. Following \eqref{lemma_init_datum_estimate_eq2}, such an $\varepsilon$-stabilization can be achieved by choosing
\begin{equation}
	T>\frac{\sigma_2}{\varepsilon^{\tilde{N}N}}\mbox{dist}(\Phi_0,\mathscr{Z})
\end{equation}
and $Nt\in\mathbb{N}\setminus \left\{0,1\right\}$ large enough. Set $\Delta t\coloneqq \frac{T}{Nt-1}$. The discretized state is $\left(\Phi_i\right)_{i=0,\dots,Nt-1}$, whereas the discretized control (velocity) is\\
$\left(\psi_{i}\right)_{i=0,\dots,Nt-2}$. The discretized functional reads as
\begin{equation}\label{discretized_functional_gen}
\widetilde{K_d}\left(\psi,\Phi\right) \coloneqq \Delta t\sum_{i=0}^{Nt-1}\left[\frac12 \|\psi_i\|^2+Q\left(\Phi_i\right)\right],
\end{equation}
subject to the discretized state equation
\begin{equation}\label{discretized_state_equation}
\frac{\Phi_i-\Phi_{i-1}}{\Delta t}=\psi_{i-1},\hspace{0.3 cm}i=1,\dots,Nt-1.
\end{equation}
\subsubsection{Algorithm execution}
\label{subsubsec:2.5.2}

By \eqref{discretized_state_equation} and \eqref{discretized_functional_gen}, the discretized minimization problem is
\begin{equation}
\mbox{minimize}\hspace{0.3 cm}\widetilde{K_d},\hspace{0.3 cm}\mbox{subject to \eqref{discretized_state_equation}}.
\end{equation}
We address the above minimization problem by employing the interior-point optimization routine \verb!IPOpt! (see \cite{IDO} and \cite{waechter2009introduction}) coupled with \verb!AMPL! \cite{FAP}, which serves as modelling language and performs the automatic differentiation. The interested reader is referred to \cite[Chapter 9]{TAT} and \cite{SMA} for a survey on existing numerical methods to solve an optimal control problem.

In figures \ref{optcontrollab1}, \ref{optcontrollab2}, \ref{optcontrollab3} and \ref{optcontrollab4}, we plot the computed optimal trajectory for \eqref{functional}, with initial datum $\Phi_0=\left(\alpha_{0,1},\gamma_{0,1};\alpha_{0,2},\gamma_{0,2}\right)\coloneqq \left(2.6,0.6, 2.5,1.5\right)$. We choose $F$, $N$ and $m_i$ (see table \ref{notation_table}), such that the condition \eqref{prop_group_eq4} is fulfilled. The exponential stabilization proved in Proposition \ref{prop_group} emerges. In figure \ref{systemresponse}, we depict the imbalance indicator versus time, along the computed trajectories. As expected, it decays to zero exponentially.


\begin{figure}[htp]
	\begin{center}
		\includegraphics[width=11cm]{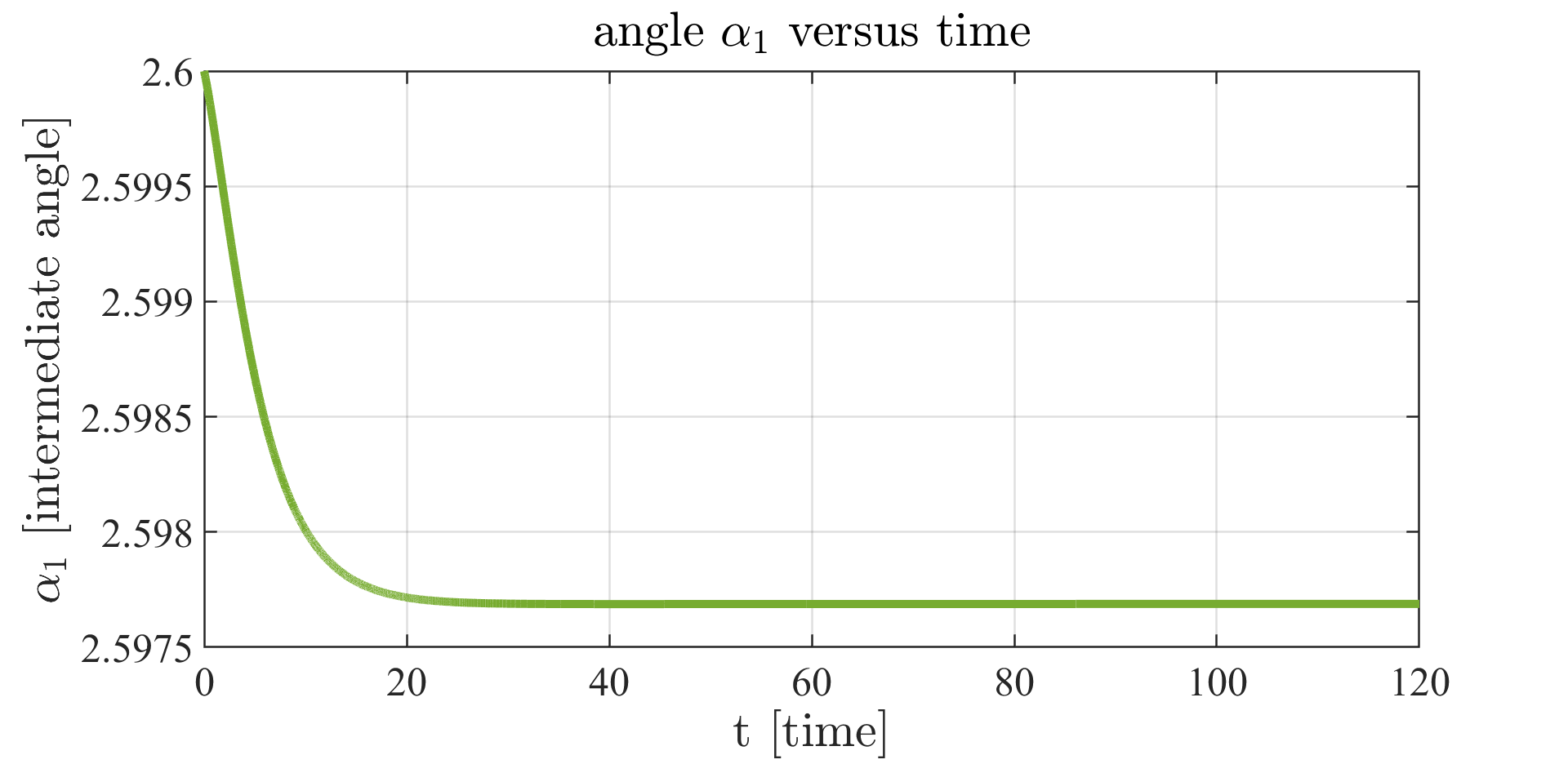}\\
		\caption{Intermediate angle $\alpha_1$ versus time}
		\label{optcontrollab1}
	\end{center}
\end{figure}

\begin{figure}[htp]
	\begin{center}
		\includegraphics[width=11cm]{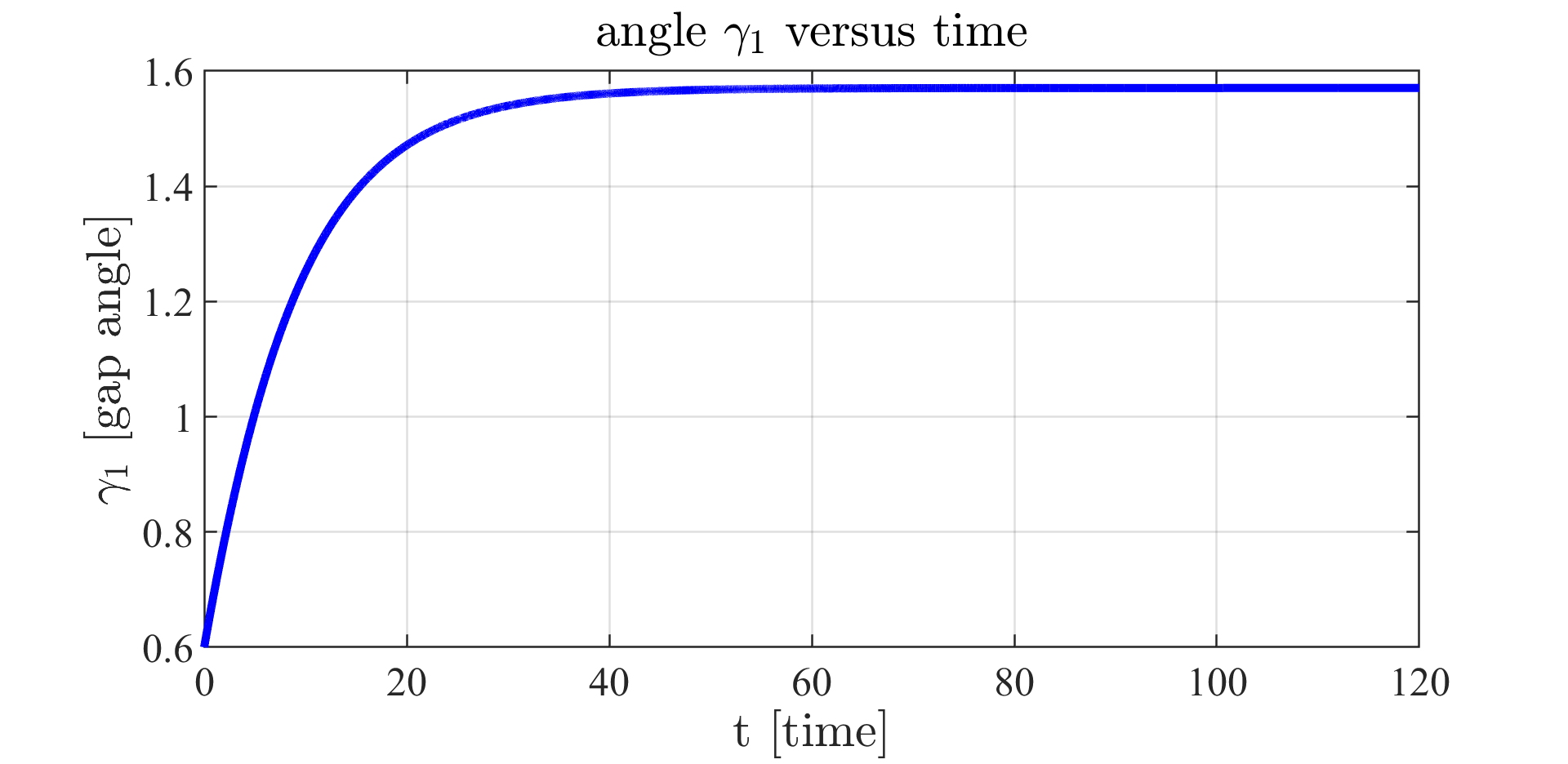}\\
		\caption{Gap angle $\gamma_1$ versus time}
		\label{optcontrollab2}
	\end{center}
\end{figure}

\begin{figure}[htp]
	\begin{center}
		\includegraphics[width=11cm]{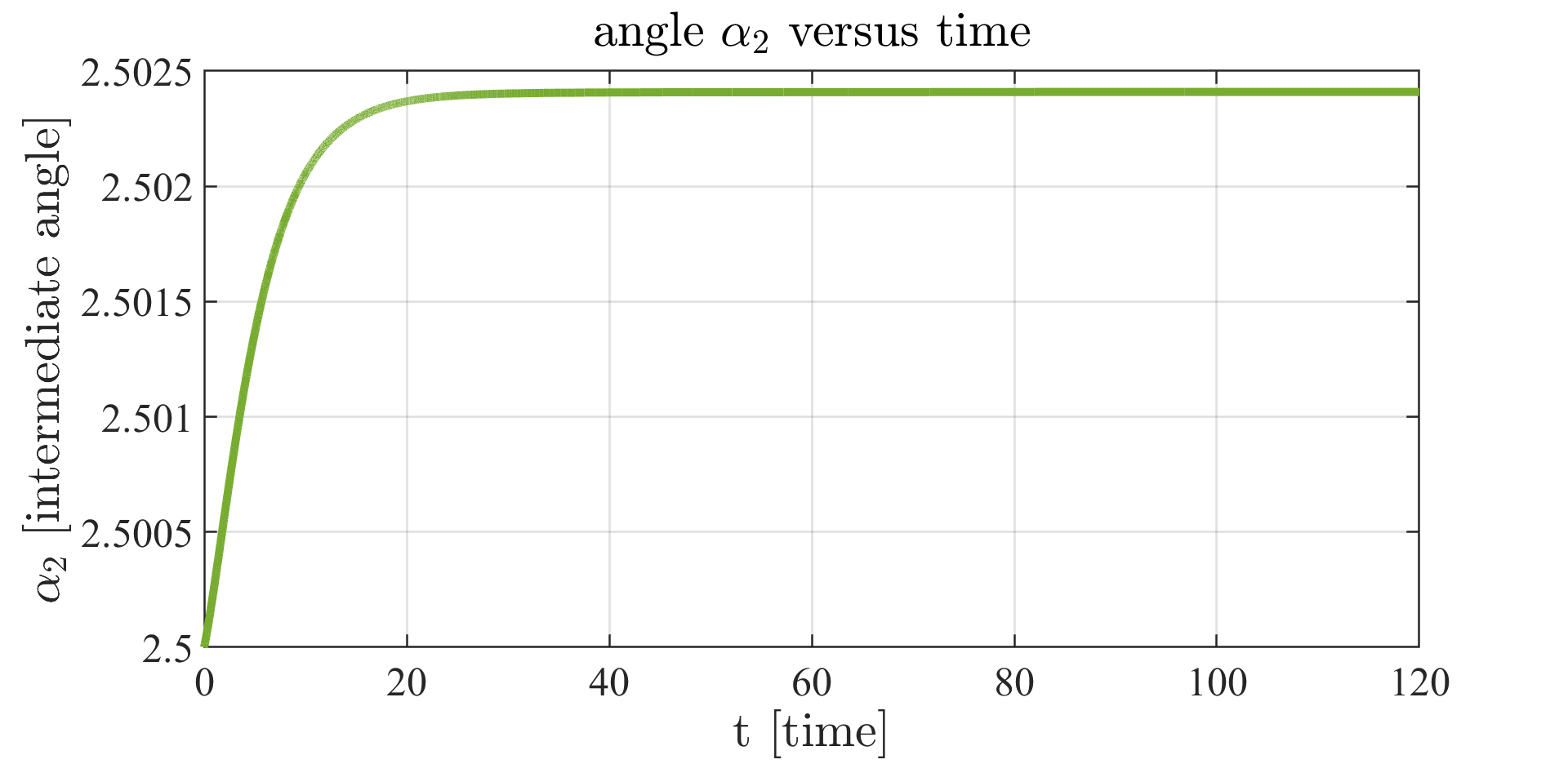}\\
		\caption{Intermediate angle $\alpha_2$ versus time}
		\label{optcontrollab3}
	\end{center}
\end{figure}

\begin{figure}[htp]
	\begin{center}
		\includegraphics[width=11cm]{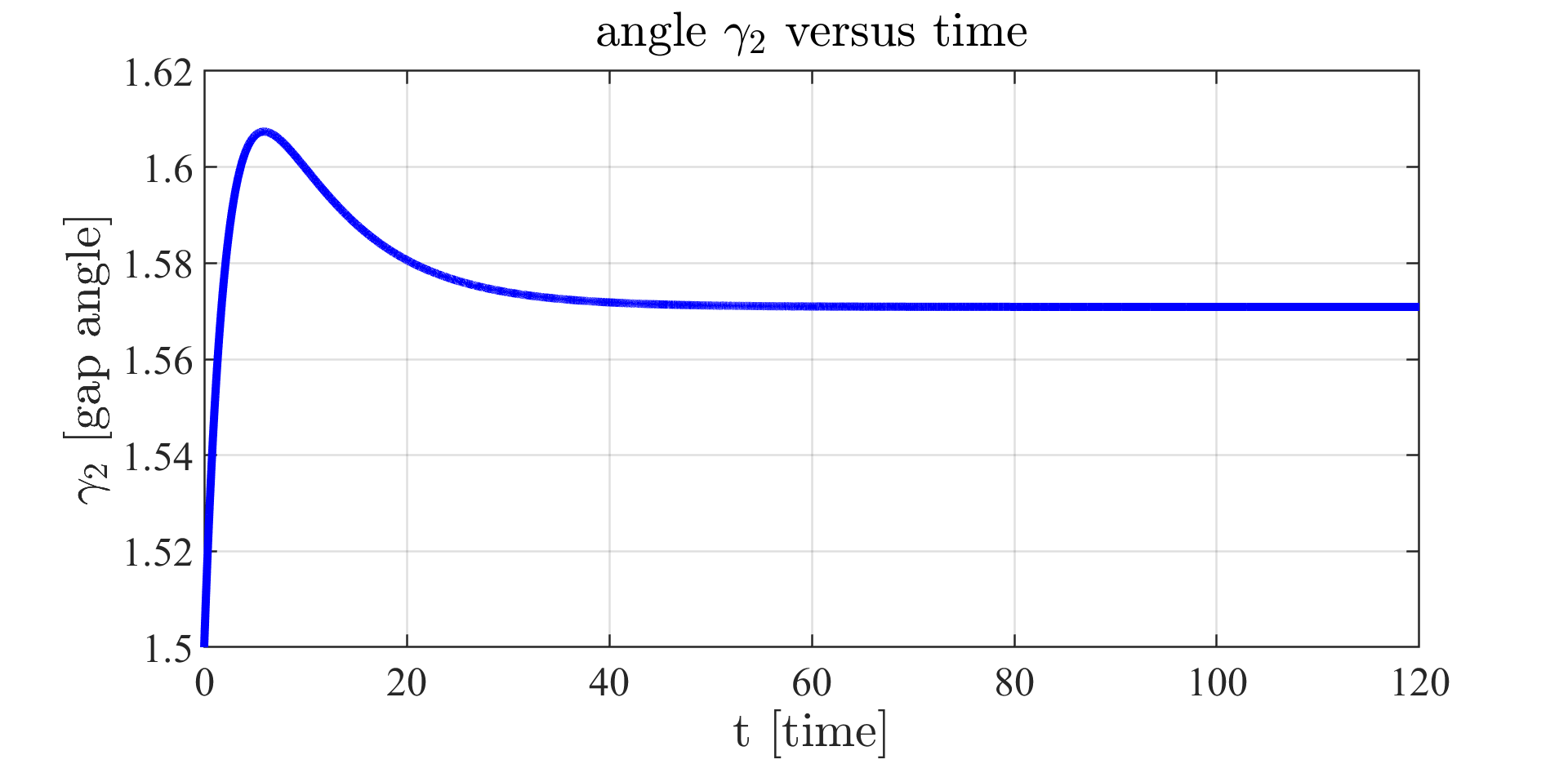}\\
		\caption{Gap angle $\gamma_2$ versus time}
		\label{optcontrollab4}
	\end{center}
\end{figure}

\begin{figure}[htp]
	\begin{center}
		\includegraphics[width=11cm]{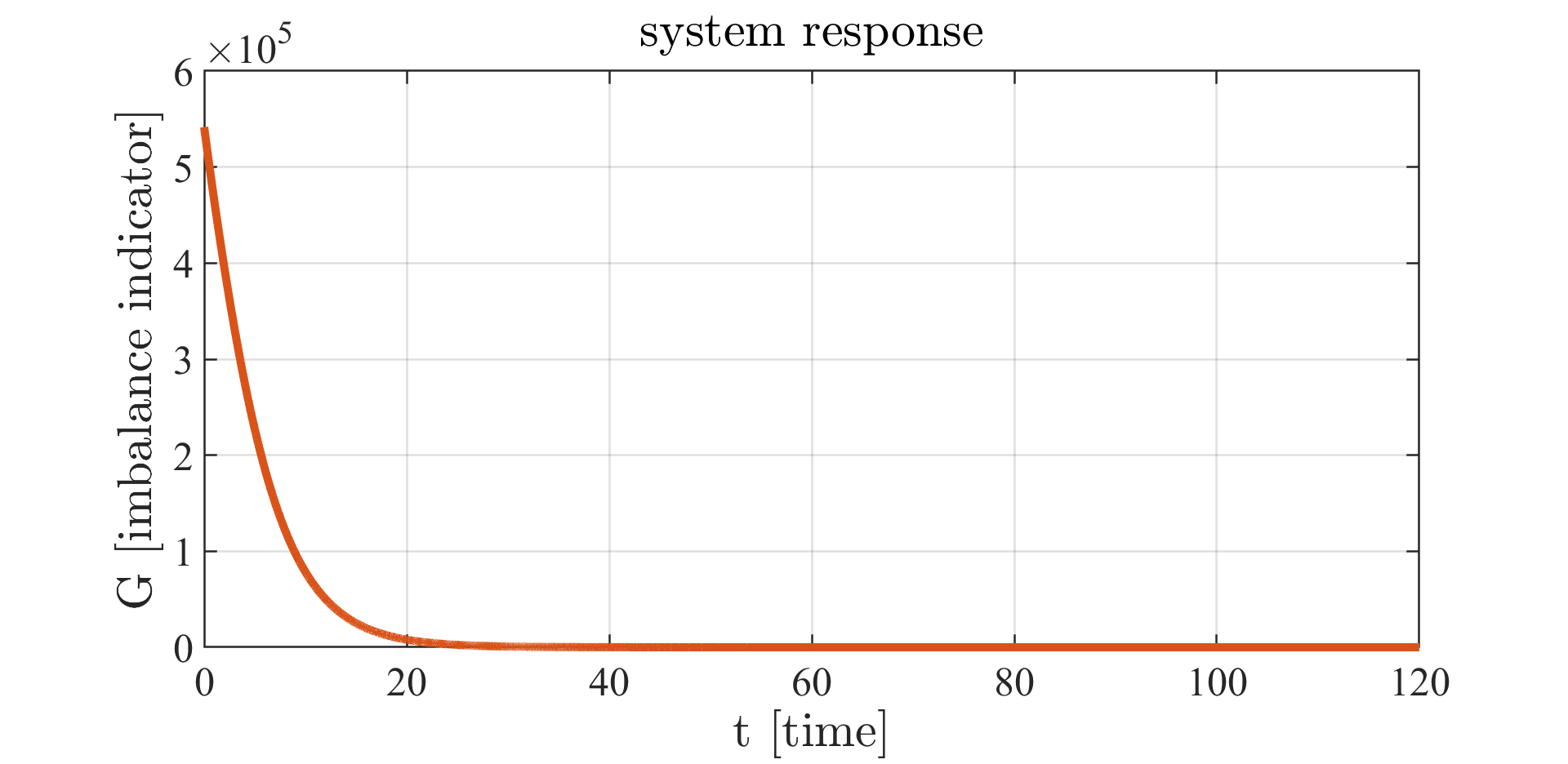}\\
		\caption{The imbalance indicator $G$ along the computed trajectory versus time.}
		\label{systemresponse}
	\end{center}
\end{figure}

\clearpage

\section{Reinforcement Learning and the closed-loop solution}
\label{sec:3}

So far, we presented an open-loop control strategy. The purpose of this section is to introduce a feedback strategy obtained by Reinforcement Learning \cite{RDP,sutton2011reinforcement,bokanowski2015value}. All throughout the section, we will work in the context of section \ref{subsec:2.3}, under the assumption $\mathscr{Z}\neq \varnothing$ and $Q$ real analytic. We start by defining the value function
\begin{equation}
V\left(\theta\right)=\inf_{\mathscr{A}_{\tiny{\theta}}}J=\inf_{\mathscr{A}_{\tiny{\theta}}}\int_0^{\infty} \frac{1}{2}\|\dot{\Phi}\|^2+Q(\Phi) \ dt,
\end{equation}
with
\begin{equation}\label{admissible_traj_theta}
\mathscr{A}_{\tiny{\theta}}\coloneqq \Bigl\{\Phi\in H^1_{loc}([0,+\infty);\mathbb{T}^n) \hspace{0.3 cm} \big| \hspace{0.3 cm} \Phi(0)=\theta\Bigr.
\end{equation}
\begin{equation}
\Bigl.\mbox{and}\hspace{0.3 cm}L(\Phi,\dot{\Phi})\in L^1(0,+\infty)\Bigr\}.
\end{equation}

Reasoning as in \cite[Theorem 6.4.8 section 6.4]{cannarsa2004semiconcave}, an optimal trajectory can be obtained by solving the (closed-loop) Ordinary Differential Equation
\begin{equation}
	\dot{\Phi}(t) = -\nabla V\left(\Phi(t)\right),\hspace{0.3 cm}t>0,
\end{equation}
with initial condition $\Phi(0)=\theta$. The feedback law is $-\nabla V\left(\Phi\right)$. Hence, the main task is to determine the value function $V$. We will show how to do this by using a value iteration algorithm of Reinforcement Learning. The convergence will be guaranteed by choosing an initial guess suggested by the stabilization/turnpike phenomenon \eqref{prop_group_eq1}. Let us mention that another approach will be to employ the analytical approximation methods developed in \cite{SMA}, which are designed to determine directly $\nabla V$.

To that end, let us write the Dynamic Programming Principle (DPP) for $V$ forward in time (see e.g. \cite[Lemma 4.3]{THJ}). Arbitrarily fix $\delta>0$:
\begin{equation}\label{eq: dynamic programming}
V\left(\theta\right) = \inf_{\Phi\in \mathscr{A}_{\tiny{\theta}}} \left\{\int_{0}^{\delta} \frac{1}{2}\|\dot{\Phi}\|^2+Q(\Phi) \ ds + V\left(\Phi(\delta)\right)\right\},
\end{equation}
where $\theta$ is an arbitrary initial configuration and $\mathscr{A}_{\tiny{\theta}}$ has been defined in \eqref{admissible_traj_theta}.

We will approximate $V$ as limit of a recursive sequence defined as
\begin{equation}\label{eq: rec_eq_prel}
V_{i+1}\left(\theta\right) = \inf_{\Phi\in \mathscr{A}_{\tiny{\theta}}} \left\{\int_{0}^{\delta} \frac{1}{2}\|\dot{\Phi}\|^2+Q(\Phi) \ ds + V_{i}\left(\Phi(\delta)\right)\right\},
\end{equation}
with $i\in \mathbb{N}$ and $\theta \in \mathbb{T}$. As announced, the choice of the initial guess is a key point.

\subsection{Initial guess}
\label{subsec:3}

We start by constructing an initial guess $V_0\in C^{0}\left(\mathbb{T}^n\right)$.

Define
\begin{equation}
	V_0:\mathbb{T}^n\longrightarrow \mathbb{R}^+,
\end{equation}
\begin{equation}\label{initial_guess_V0}
	V_0\left(\theta\right)\coloneqq \frac12 \mbox{dist}\left(\theta,\mathscr{Z}\right)^2+\frac{L}{2}\mbox{dist}\left(\theta,\mathscr{Z}\right),
\end{equation}
with $L=\max_{\mathbb{T}^n}\left\|\nabla Q\right\|$.

\subsection{Approximating sequence}
\label{subsec:3.2}

Let us define a sequence $\left\{V_i\right\}_{i\in \mathbb{N}}\subset L^{\infty}\left(\mathbb{T}^n\right)$ approximating $V$. The initial guess has been defined in \eqref{initial_guess_V0}, while for any $i\in \mathbb{N}$ we set
\begin{equation}\label{eq: rec_eq}
V_{i+1}\left(\theta\right) \coloneqq \inf_{\Phi\in \mathscr{A}_{\tiny{\theta}}} \left\{\int_{0}^{\delta} \frac{1}{2}\|\dot{\Phi}\|^2+Q(\Phi) \ ds + V_{i}\left(\Phi(\delta)\right)\right\},
\end{equation}
for $\theta \in \mathbb{T}$.

\subsection{Convergence of the algorithm}
\label{subsec:3.3}

By using the stabilization/turnpike estimate \eqref{lemma_init_datum_estimate_eq2}, we prove the convergence of the algorithm.

\begin{proposition}\label{prop_alg_conv}
	In the notation of Lemma \ref{lemma_init_datum_estimate}, suppose $\mathscr{Z}\neq \varnothing$ and $Q$ is real analytic. For any $\varepsilon >0$ and for every
	\begin{equation}\label{cond_on_i}
		i> \frac{2\pi \sqrt{n}\sigma_2}{\delta\left(\frac{-L+\sqrt{L^2+8\varepsilon}}{2}\right)^{\tilde{N}N}}
	\end{equation}
	we have
	\begin{equation}
		\left|V_i\left(\theta\right)-V\left(\theta\right)\right|<\varepsilon, \hspace{0.3 cm}\forall \ \theta\in \mathbb{T}.
	\end{equation}
\end{proposition}
\begin{proof}[Proof of Proposition \ref{prop_alg_conv}]
	\textit{Step 1} \  \textbf{Prove $V\left(\theta\right)\leq V_0\left(\theta\right)$, for any $\theta \in \mathbb{T}^n$.}\\
	Let $\overline{\theta}\in \mathscr{Z}$, such that $\left\|\theta-\overline{\theta}\right\|=\mbox{dist}(\theta,\mathscr{Z})$. Consider the trajectory
	\begin{equation}\widehat{\Phi}(t)\coloneqq
	\begin{dcases}
	(1-t)\theta+t\overline{\theta} \hspace{0.3 cm} &t\in [0,1)\\
	\overline{\theta} \hspace{0.3 cm} &t\in [1,+\infty).\\
	\end{dcases}
	\end{equation}
Proceeding as in Step 1 of the proof of Lemma \ref{lemma_init_datum_estimate}, we get
	\begin{equation}\label{}
	V\left(\theta\right)\leq K\left(\widehat{\Phi}\right)\leq \mbox{dist}(\Phi_0,\mathscr{Z})^2+\frac{L}{2}\mbox{dist}(\Phi_0,\mathscr{Z})=V_0\left(\theta\right).
	\end{equation}
	\textit{Step 2} \  \textbf{Prove $V\left(\theta\right)\leq V_i\left(\theta\right)$, for each $i\in \mathbb{N}$ and for any $\theta \in \mathbb{T}^n$.}\\
	We proceed by induction on $i\in \mathbb{N}$. By step 1, the assertion holds for $i=0$. Let us assume
	\begin{equation}
		V\left(\theta\right)\leq V_i\left(\theta\right), \hspace{0.3 cm}\forall \ \theta \in \mathbb{T}^n
	\end{equation}
	and prove
	\begin{equation}
		V\left(\theta\right)\leq V_{i+1}\left(\theta\right), \hspace{0.3 cm}\forall \ \theta \in \mathbb{T}^n.
	\end{equation}
	By definition, for any $\eta > 0$, there exists $\Phi_{\eta}\in \mathscr{A}_{\tiny{\theta}}$ such that
	\begin{equation}\label{prop_alg_conv_eq6}
		V_{i+1}\left(\theta\right) + \eta > \left\{\int_{0}^{\delta} \frac{1}{2}\|\dot{\Phi}_{\eta}\|^2+Q\left(\Phi_{\eta}\right) \ ds + V_{i}\left(\Phi_{\eta}(\delta)\right)\right\}.
	\end{equation}
	Now, by induction assumption,
	\begin{align}\label{prop_alg_conv_eq8}
	&\left\{\int_{0}^{\delta} \frac{1}{2}\|\dot{\Phi}_{\eta}\|^2+Q\left(\Phi_{\eta}\right) \ ds + V_{i}\left(\Phi_{\eta}(\delta)\right)\right\} \nonumber\\
	&\geq\left\{\int_{0}^{\delta} \frac{1}{2}\|\dot{\Phi}_{\eta}\|^2+Q\left(\Phi_{\eta}\right) \ ds + V\left(\Phi_{\eta}(\delta)\right)\right\},
	\end{align}
	whence, by \eqref{prop_alg_conv_eq6}, we obtain
	\begin{equation}\label{}
	V_{i+1}\left(\theta\right) + \eta > \left\{\int_{0}^{\delta} \frac{1}{2}\|\dot{\Phi}_{\eta}\|^2+Q\left(\Phi_{\eta}\right) \ ds + V\left(\Phi_{\eta}(\delta)\right)\right\}\geq V\left(\theta\right),
	\end{equation}
	where in the last inequality we have used the Dynamic Programming Principle (DPP) \eqref{eq: dynamic programming}. The arbitrariness of $\eta$ allows to conclude this step.
	
	\textit{Step 3} \  \textbf{For any $\varepsilon >0$ there exists $t_{\varepsilon}>0$ such that
	\begin{equation}\label{V_dist}
		\left|V_i\left(\Phi_{\theta}(t)\right)-V\left(\Phi_{\theta}(t)\right)\right|< \varepsilon,
	\end{equation}
	for any $t> t_{\varepsilon}-i\delta$, for each $i\in \mathbb{N}$ and $\theta \in \mathbb{T}$.}\\
	In the above expression and in the remainder of the proof, $\Phi_{\theta}$ denotes an optimal trajectory for \eqref{functional_gen} with initial configuration $\theta$. Set
	\begin{equation}\label{t_eps}
		t_{\varepsilon}\coloneqq \frac{2\pi \sqrt{n}\sigma_2}{\left(\frac{-L+\sqrt{L^2+8\varepsilon}}{2}\right)^{\tilde{N}N}}.
	\end{equation}
	We prove the assertion by induction. Let us start with $i=0$. By \eqref{lemma_init_datum_estimate_eq2}, for any $t> t_{\varepsilon}$ we have
	\begin{equation}
	\mbox{dist}\left(\Phi(t),\mathscr{Z}\right)<\frac{-L+\sqrt{L^2+8\varepsilon}}{2},
	\end{equation}
	whence
	\begin{equation}
	V_0\left(\theta\right)=\frac12 \mbox{dist}\left(\theta,\mathscr{Z}\right)^2+\frac{L}{2}\mbox{dist}(\Phi_0,\mathscr{Z})<\varepsilon.
	\end{equation}
	We suppose the assertion for $i$ and we prove it for $i+1$. By definition \eqref{eq: rec_eq}, we have
	\begin{equation}
	V_{i+1}\left(\Phi_{\theta}(t)\right)\leq \int_{t}^{t+\delta} \frac{1}{2}\|\dot{\Phi}_{\theta}\|^2+Q\left(\Phi_{\theta}\right) \ ds + V_{i}\left(\Phi_{\theta}(t+\delta)\right).
	\end{equation}
	Then, for any $t> t_{\varepsilon}-(i+1)\delta$
	\begin{align}
	\left|V_{i+1}\left(\Phi_{\theta}(t)\right)-V\left(\Phi_{\theta}(t)\right)\right|&= V_{i+1}\left(\Phi_{\theta}(t)\right)-V\left(\Phi_{\theta}(t)\right)\label{prop_alg_conv_eq20_sub1}\\
	&\leq\int_{t}^{t+\delta} \frac{1}{2}\|\dot{\Phi}_{\theta}\|^2+Q\left(\Phi_{\theta}\right) \ ds + V_{i}\left(\Phi_{\theta}(t+\delta)\right)\nonumber\\
	&\hspace{0.46 cm}-V\left(\Phi_{\theta}(t)\right)\nonumber\\
	&=\int_{t}^{t+\delta} \frac{1}{2}\|\dot{\Phi}_{\theta}\|^2+Q\left(\Phi_{\theta}\right) \ ds + V_{i}\left(\Phi_{\theta}(t+\delta)\right)\nonumber\\
	&\hspace{0.46 cm}-\int_{t}^{t+\delta} \frac{1}{2}\|\dot{\Phi}_{\theta}\|^2+Q\left(\Phi_{\theta}\right) \ ds\label{prop_alg_conv_eq20_sub6}\\
	&\hspace{0.46 cm}-V\left(\Phi_{\theta}\left(t+\delta\right)\right)\nonumber\\
	&=V_{i}\left(\Phi_{\theta}(t+\delta)\right)-V\left(\Phi_{\theta}\left(t+\delta\right)\right)\nonumber\\
	&<\varepsilon\label{prop_alg_conv_eq20_sub8},\\
	\end{align}
	where in \eqref{prop_alg_conv_eq20_sub1} we have employed step 2, in \eqref{prop_alg_conv_eq20_sub6}  the Dynamic Programming Principle (DPP) \eqref{eq: dynamic programming} and in \eqref{prop_alg_conv_eq20_sub8} the induction assumption together with $t+\delta> t_{\varepsilon}-i\delta$.
	
	\textit{Step 4} \  \textbf{Conclusion.}\\
	For any $i\in \mathbb{N}$ satisfying \eqref{cond_on_i}, we have $t_{\varepsilon}-i\delta <0$, where $t_{\varepsilon}$ has been defined in \eqref{t_eps}. Then, we apply \eqref{V_dist} with $t=0$ thus concluding.
\end{proof}

\section{Conclusions and perspectives}
\label{sec:4}

In this manuscript, a problem of rotors imbalance suppression has been addressed. A physical model has been conceived. The control problem has been formulated in the context of the Calculus of Variations, in an infinite time horizon. A general class of variational problems has been introduced, containing imbalance suppression as a particular case. In this general framework, well-posedness in infinite-time has been proved and Optimality Conditions have been derived both as second order Euler-Lagrange equations and first order Pontryagin system. The \L ojasiewicz inequality has been employed to prove convergence of the time optima towards the steady optima. Quantitative estimates of the rate of convergence have been obtained, without sign condition on the hessian of the imbalance indicator. In case the imbalance is below a given threshold, Stable Manifold theory has been used to obtain an exponential estimate of the speed of convergence. In case real-time feedback on the imbalance is available, a value iteration Reinforcement Learning algorithm has been proposed.

Both open-loop and feedback optimal controls have been designed. In the case of closed-loop, our Reinforcement Learning algorithm can be complemented by Hamilton-Jacobi theory (see e.g. \cite{SMA,bardi2008optimal}). The Hamilton-Jacobi equation for our functional \eqref{functional_gen} reads as
\begin{equation}\label{HJ}
\left\|\nabla V\left(\theta\right)\right\|^2=2Q\left(V\left(\theta\right)\right)\hspace{0.6 cm} \theta\hspace{0.03 cm}\in\mathbb{T}^n,
\end{equation}
where
\begin{equation}
V\left(\theta\right)=\inf_{\mathscr{A}_{\tiny{\theta}}}J=\inf_{\mathscr{A}_{\tiny{\theta}}}\int_0^{\infty} \frac{1}{2}\|\dot{\Phi}\|^2+Q(\Phi) \ dt,
\end{equation}
with
\begin{equation}
\mathscr{A}_{\tiny{\theta}}\coloneqq \Bigl\{\Phi\in H^1_{loc}([0,+\infty);\mathbb{T}^n) \hspace{0.3 cm} \big| \hspace{0.3 cm} \Phi(0)=\theta\Bigr.
\end{equation}
\begin{equation}
\Bigl.\mbox{and}\hspace{0.3 cm}L(\Phi,\dot{\Phi})\in L^1(0,+\infty)\Bigr\}.
\end{equation}
As we have seen in section \ref{sec:3}, the value function $V$ can be approximated numerically by a value iteration algorithm of Reinforcement Learning
. Another approach could be to employ numerical solvers for the Hamilton-Jacobi equation like ROC-HJ \cite{bokanowski2019user}. Furthermore, we could employ the analytical methods illustrated in \cite{SMA}, whose goal is to approximate directly $\nabla V$.
	
\section*{Appendix}
The appendix is devoted to the proof of Lemma \ref{lemma_argming}, Proposition \ref{prop3_gen} and Proposition \ref{prop_EL_no_final_condition}.

\section{Proof of Proposition \ref{prop3_gen}}

Now, we prove the well posedness of the time-evolution problem, by employing the direct methods in the Calculus of Variations.
\begin{proof}[Proof of Proposition \ref{prop3_gen}.]
	\textit{Step 1} \ \textbf{Boundedness of the minimizing sequence.}\\
	
	Let $\left\{\Phi^m\right\}_{m\in \mathbb{N}}\subset \mathscr{A}$ be a minimizing sequence for \eqref{functional_gen}. We wish to prove that $\left\{\dot{\Phi}^m\right\}_{m\in \mathbb{N}}\subset L^2((0,+\infty);\mathbb{R}^n)$ is bounded.
	
	By definition of minimizing sequence, if $m$ is large enough,
	\begin{equation}
	\frac{1}{2}\|\dot{\Phi}^m\|_{L^2}^2\leq K(\Phi^m)\leq \inf_{\mathscr{A}}K+1.
	\end{equation}
	Then, $\|\dot{\Phi}^m\|_{L^2}\leq M$ for any natural $m$, as desired.\\
	\textit{Step 2} \ \textbf{Weak convergence of the minimizing sequence in $\mathscr{A}$.}\\
	Now, for any $t\geq 0$,
	\begin{equation}
	\Phi^m(t)=\Phi_0+\int_0^t \dot{\Phi}^m(s)ds.
	\end{equation}
	Then, by Cauchy-Schwarz inequality, for any $T>0$, $\|\Phi^m\|_{L^{2}((0,T);\mathbb{T}^n)}\leq M\left(\sqrt{T}+1\right)$. Hence, by Banach-Alaoglu Theorem, there exists $\Phi\in H^1_{loc}((0,+\infty);\mathbb{T}^n)$ with $\dot{\Phi}\in L^2((0,+\infty);\mathbb{R}^n)$ such that, up to subsequences,
	\begin{equation}
	\Phi^m\underset{m\to \infty}{\longrightarrow}\Phi
	\end{equation}
	weakly in $H^1((0,T);\mathbb{T}^n)$ for any $T>0$ and
	\begin{equation}
	\dot{\Phi}^m\underset{m\to \infty}{\longrightarrow}\dot{\Phi},
	\end{equation}
	weakly in $L^2((0,+\infty);\mathbb{R}^n)$. Furthermore, the above convergence occurs point-wise. Indeed, for $t\geq 0$ and $T\geq t$, the linear operator
	\begin{equation}
	\delta_{t}:H^1((0,T);\mathbb{T}^n)\longrightarrow \mathbb{R}^n
	\end{equation}
	\begin{equation}
	\Phi\longmapsto \Phi(t)
	\end{equation}
	is continuous. Hence, by the definition of weak convergence,\\
	$\Phi^m(t)=\delta_t(\Phi^m)\longrightarrow \delta_{t}(\Phi)=\Phi(t)$. Since, for any natural $m$, $\Phi^m(0)=\Phi_0$, we have $\Phi(0)=\Phi_0$, whence $\Phi\in \mathscr{A}$, as required.\\
	\textit{Step 3} \ \textbf{Conclusion}\\
	By the lower semicontinuity of the norm with respect to the weak convergence
	\begin{equation}\label{lsc_dot}
	\int_0^{\infty}\|\dot{\Phi}\|^2 dt\leq \liminf_{m\to +\infty}\int_0^{\infty}\|\dot{\Phi}^m\|^2 dt.
	\end{equation}
	
	At this stage, we want to prove the inequality
	\begin{equation}\label{lsc_QPhi_old}
	\int_0^{\infty}Q(\Phi) dt\leq \liminf_{m\to +\infty}\int_0^{\infty}Q(\Phi^m) dt.
	\end{equation}
	Now, as we have shown in Step 2, $\Phi^m$ converges to $\Phi$ point-wise, whence
	\begin{equation}
	Q(\Phi^m(t))\longrightarrow Q(\Phi(t))
	\end{equation}
	for any $t\geq 0$.
	Furthermore, by Weierstrass theorem $Q:\mathbb{T}^n\longrightarrow \mathbb{R}^+$ is bounded. Then, for every $T>0$, by the Dominated Convergence Theorem,
	\begin{equation}
	Q\left(\Phi^m\right)\longrightarrow Q\left(\Phi\right)
	\end{equation}
	in the $L^1((0,T);\mathbb{R})$ norm, whence
	\begin{eqnarray*}
		\int_0^T Q\left(\Phi\right) dt&=&\lim_{m\to +\infty} \int_0^T Q\left(\Phi^m\right)dt\nonumber\\
		&\leq&\liminf_{m\to +\infty}\int_0^{\infty} Q\left(\Phi^m\right)dt.
	\end{eqnarray*}
	Hence, by arbitrariness of $T>0$,
	\begin{equation}\label{lsc_QPhi}
	\int_0^{\infty} Q\left(\Phi\right) dt\leq \liminf_{m\to +\infty}\int_0^{\infty} Q\left(\Phi^m\right)dt,
	\end{equation}
	i.e. \eqref{lsc_QPhi_old}.
	
	In conclusion, by \eqref{lsc_dot} and \eqref{lsc_QPhi}, we have
	\begin{eqnarray*}
		K(\Phi)&=&\frac{1}{2}\int_0^{\infty}\|\dot{\Phi}\|^2+Q\left(\Phi\right) dt\nonumber\\
		&\leq&\liminf_{m\to +\infty}\frac{1}{2}\int_0^{\infty}\|\dot{\Phi}^m\|^2+Q\left(\Phi^m\right) dt\nonumber\\
		&=&\inf_{\mathscr{A}}K,\nonumber\\
	\end{eqnarray*}
	whence $\Phi\in \mathscr{A}$ is the required minimizer. This finishes the proof.
\end{proof}

\section{Proof of Proposition \ref{prop_EL_no_final_condition}}

After proving the existence of minimizers for \ref{functional_gen}, we derive the Optimality Conditions.

\begin{proof}[Proof of Proposition of \ref{prop_EL_no_final_condition}]
	\textit{Step 1} \  \textbf{Regularity of $\Phi$ by the fundamental Lemma of the Calculus of Variations}\\
	Take $\Phi$ a minimizer of \eqref{functional_gen}. By \eqref{differential_gen} and Fermat's theorem, for any direction $v\in C^{\infty}_c((0,+\infty);\mathbb{R}^n)$, we have
	\begin{equation}\label{Fermat_cond_1}
	\int_0^{\infty} \dot{\Phi} \dot{v}+\nabla Q\left(\Phi\right)\cdot vdt=\langle dK(\Phi),v\rangle=0.
	\end{equation}
	Then, by the fundamental Lemma in the Calculus of Variations (see \cite{hobson1913fundamental}),
	$\Phi\in C^2([0,T];\mathbb{T}^n)$.\\
	\textit{Step 2} \  \textbf{Proof of (2)}\\
	Since $\Phi\in C^2$, we are allowed to integrate by parts in \eqref{Fermat_cond_1}, getting
	\begin{eqnarray*}
		0&=&\int_0^{\infty} \dot{\Phi} \dot{v}+\nabla Q\left(\Phi\right) vdt\nonumber\\
		&=&\lim_{T\to +\infty}\dot{\Phi}(T)v(T)+\int_0^{\infty}\left[-\ddot{\Phi}+ \nabla Q\left(\Phi\right) \right]v \hspace{0.03 cm} dt,
	\end{eqnarray*}
	which, thanks to the arbitrariness of $v$, leads to the differential equation in \eqref{EL_no_final_condition}. Furthermore, by bootstrapping in \eqref{EL_no_final_condition}, we have the $C^{\infty}$ regularity of the minimizer $\Phi$.\\
	\textit{Step 3} \  \textbf{Proof of (3)}\\
	Consider the energy
	\begin{equation}
	E(t)=\frac12\|\dot{\Phi}(t)\|^2-Q\left(\Phi(t)\right)
	\end{equation}
	and, take the time derivative
	\begin{equation}\label{energy_derivative_computation}
	\dot{E}(t)=\dot{\Phi}(t)\cdot \ddot{\Phi}(t)-\nabla Q\left(\Phi(t)\right)\cdot \dot{\Phi}(t) = -\left[-\ddot{\Phi}+ \nabla Q\left(\Phi\right)\right]\cdot \dot{\Phi}(t) =0,
	\end{equation}
	where in the last equality we have used the differential equation in \eqref{EL_no_final_condition}. Now, the integral $\int_0^{\infty}\|\dot{\Phi}\|^2 dt$ is finite, whence there exists a sequence $\left\{T_{q}\right\}\subset (0,+\infty)$ such that $T_q\underset{q\to \infty}{\longrightarrow}+\infty$ and
	\begin{equation}
	\dot{\Phi}(T_q)\underset{q\to +\infty}{\longrightarrow}0 \hspace{0.3 cm}\mbox{and}\hspace{0.3 cm}Q\left(\Phi(t)\right)\underset{q\to +\infty}{\longrightarrow}0.
	\end{equation}
	Therefore, the energy $E\left(T_q\right)\underset{q\to +\infty}{\longrightarrow}0$, whence, by using \eqref{energy_derivative_computation}, we have $E(t)\equiv 0$.
	
\end{proof}

\section{General mathematical notation}
\label{sec:Notation}

The circumference is denoted by
\begin{equation}
\mathbb{T}\coloneqq{\mathbb{R}}/{\sim},
\end{equation}
where $\varphi_1\sim \varphi_2$ if and only if there exists an integer $k$ such that $\varphi_2= \varphi_1+2k\pi$.

We introduce the following function spaces:
\begin{equation}
L^2_{loc}((0,+\infty);\mathbb{R}^n)\coloneqq\bigcap_{T>0} L^2((0,T);\mathbb{R}^n).
\end{equation}
\vspace{0.4 cm}
\begin{equation}
H^1((0,T);\mathbb{T}^n)\coloneqq \Bigl\{\Phi\in L^2((0,T);\mathbb{T}^n) \ | \Bigr.
\end{equation}
\begin{equation}
\Bigl. \Phi \mbox{ is weakly differentiable and }\dot{\Phi}\in L^2((0,T);\mathbb{T}^n)\Bigr\}.
\end{equation}
\vspace{0.4 cm}
\begin{equation}\label{H1loc}
H^1_{loc}([0,+\infty);\mathbb{T}^n)\coloneqq \left\{\Phi\in H^1((0,T);\mathbb{T}^n), \hspace{0.3 cm}\forall \hspace{0.03 cm}T>0\right\};
\end{equation}
\vspace{0.4 cm}
\begin{equation}
C^{\infty}_c((0,+\infty);\mathbb{R}^n)\coloneqq \left\{\Phi:[0,+\infty)\longrightarrow \mathbb{R}^n \ |\right.
\end{equation}
\begin{equation}
\Phi \mbox{ is infinitely many times differentiable}
\end{equation}
\begin{equation}
\left. \mbox{and}\hspace{0.12 cm}\mbox{supp}(\Phi)\subset\subset (0,+\infty)\right\}.
\end{equation}

	\bibliography{my_references}
	\bibliographystyle{siam}
	
\end{document}